\documentclass[10pt,psamsfonts]{amsart}
\usepackage{amsmath}
\usepackage{amsthm}
\usepackage{amssymb}
\usepackage{amscd}
\usepackage{amsfonts}
\usepackage{amsbsy}
\usepackage{graphicx}
\usepackage[dvips]{psfrag}
\usepackage{array}
\usepackage{color}
\usepackage{epsfig}
\usepackage{url}
\usepackage{overpic}
\usepackage{tikz-cd}
\usepackage{enumitem}
\usepackage{hyperref}
\usepackage{soul}
\usepackage{faktor}
\usepackage{float}
\usepackage[normalem]{ulem}
\usepackage{nccmath}
\usepackage[foot]{amsaddr}
\usepackage{cleveref}

\usepackage{textcomp}

\newcolumntype{L}{>{\displaystyle}l}
\newcolumntype{C}{>{\displaystyle}c}
\newcolumntype{R}{>{\displaystyle}r}

\newcommand{\R}{\ensuremath{\mathbb{R}}}
\newcommand{\N}{\ensuremath{\mathbb{N}}}

\newcommand{\Z}{\ensuremath{\mathbb{Z}}}
\newcommand{\cc}{\ensuremath{\mathbb{S}}}

\newcommand{\T}{\theta}

\newcommand{\bx}{{\bf x}}

\newcommand{\Id}{\mathrm{Id}}

\newcommand{\x}{\mathbf{x}}

\newcommand{\bP}{\mathbf{P}}

\def\p{\partial}
\def\e{\varepsilon}

\newtheorem {theorem} {Theorem}
\newtheorem {definition} {Definition}

\newtheorem {corollary}{Corollary}
\newtheorem {lemma}{Lemma}

\newtheorem {remark}{Remark}

\newtheorem{wassumption}{Working Assumption}
\newtheorem {mtheorem} {Theorem}

{
	\refstepcounter{wassumption}%
	\renewcommand{\thewassumption}{#1}%
	\begin{wassumption}\label{#2}%
	}
	{
	\end{wassumption}%
}
\makeatletter
\newcommand{\nameditem}[1]{%
	\item[#1]\phantomsection\def\@currentlabel{#1}%
}
\makeatother

\usepackage{mathpazo}

\begin{document}
\renewcommand{\arraystretch}{1.5}

\title[Persistence of Weakly Normally Hyperbolic Invariant Tori]
{Weakly Normally Hyperbolic Invariant Tori:\\  Persistence and an Averaging Principle}

\author[D.D. Novaes and P.C.C.R. Pereira]
{Douglas D. Novaes$^1$  and Pedro C.C.R. Pereira$^{2,\dagger}$}

\address{Universidade Estadual de Campinas (UNICAMP), Instituto de Matemática, Estatística e Computação Científica (IMECC), Departamento de Matemática, Campinas, SP, Brasil$^1$}
\address{Universidade Estadual Paulista (UNESP), Instituto de Biociências, Letras e Ciências Exatas, São José do Rio Preto, SP, Brasil$^2$}
\email{ddnovaes@unicamp.br$^1$}
\email{pccr.pereira@unesp.br$^2$}

\keywords{invariant tori, normally hyperbolic invariant manifolds, singular perturbations, averaging theory, polynomial vector fields}

\subjclass[2020]{34C23, 34C29, 34C45, 37C70}
\thanks{$^\dagger$ Corresponding author}

\begin{abstract}
We prove a persistence theorem for attracting weakly normally hyperbolic invariant tori under small time-periodic perturbations. The theorem extends recent continuation results for weakly normally hyperbolic limit cycles to invariant tori of arbitrary dimension, providing a general analytical framework for their detection. As an application, we establish an averaging principle showing that attracting normally hyperbolic invariant $d$-tori of the averaged system give rise to attracting normally hyperbolic invariant $(d+1)$-tori of the original non-autonomous system. We further introduce a polynomiality-preserving construction that simultaneously lifts the dimensions of the phase space and the attracting normally hyperbolic invariant tori, yielding recursive lower bounds for the maximal number of codimension-$1$ normally hyperbolic invariant tori of polynomial vector fields of a given degree, thereby extending the counting aspect of Hilbert's sixteenth problem to higher-dimensional invariant tori. In particular, we prove that this number grows at least polynomially with the degree and becomes unbounded for invariant tori of higher codimension.
\end{abstract}

\maketitle


\section{Introduction and statement of results}

Identifying attracting sets of a dynamical system allows for the prediction of its long-term behavior, even when the precise internal dynamics remain unresolved. Consequently, the study of these objects is of fundamental importance, especially at the intersection of dynamical systems theory and the applied sciences (see, e.g., \cite{field1974oscillations,hutson1984theorem,taylor1978evolutionarily}). Attracting invariant tori are of particular interest due to their intrinsic connection to periodic and quasi-periodic dynamics (see, e.g., \cite{GREBOGI1985354}). However, because rigorous analytical detection is notoriously difficult, their existence in applied systems is often inferred exclusively through numerical methods (see, e.g., \cite{KUZNETSOV20101676,KKSS19}). 

In this paper, we introduce a new analytical tool to detect invariant $d$-dimensional tori---hereafter simply called $d$-tori---for a class of differential systems. This tool is derived by investigating what is best described as the problem of ``persistence of weakly normally hyperbolic manifolds''. There is no single precise formulation capturing it in its full generality; rather, it comprises a collection of different results dealing with extensions of general theorems on persistence of invariant manifolds to settings where they are not directly applicable due to the singular nature of the family of differential systems. We present a brief historical overview of the problem as it relates to our goal before formalizing it within the specific framework adopted in this work.

Its roots can be traced back to the 1970's, when theorems on persistence of \textit{normally hyperbolic invariant manifolds} under regular perturbations of the vector field were proved independently by Fenichel \cite{Fenichel1971} and Hirsch, Pugh, and Shub \cite{hirschpughshub}. In essence, normal hyperbolicity guarantees that the dynamics normal to the invariant manifold dominate the tangential dynamics, and this property proved to be fundamental in establishing their robustness with respect to perturbations of the system. This persistence property was used to investigate systems with multiple time-scales by Fenichel himself \cite{fenichel1979geometric}, originating the fruitful framework of Geometric Singular Perturbation Theory (see \cite{wechselberger2020geometric} for more details).

Efforts to find analogous persistence results for ``weaker'' forms of normal hyperbolicity then appear to have been initiated by Kopell \cite{Kopell85} (see also \cite[Section 7.3]{wigginsnormally}), motivated by an atmospheric model. However, as pointed out by Chicone and Liu in \cite[Section 3]{MR1740943}, the continuation method originally devised by Kopell contained a gap, which was subsequently rectified in the same reference, albeit only for a specific family of planar vector fields and with the substantial restriction of exclusively considering limit cycles as the invariant manifolds whose persistence can be asserted. The form of this specific family was recently generalized in \cite{PNC23}, whose main achievement was the development of a tool to detect $2$-tori in three-dimensional systems. This tool was employed in \cite{NV26} to analytically confirm the tori observed in \cite{KKSS19}.

As to the problem itself, let us follow the historical development and first discuss the original persistence theorems for regular perturbations. Consider an open set $U \in \R^n$ and suppose you have a perturbative family of differential equations of the form
\begin{equation} \label{system.statement.fenichel}
	\dot x = f(x) + \e h(x,\e),
\end{equation}
with $f$ and $h$ of class $C^1$ over $U \times (-\e_0,\e_0)$, $\e_0>0$. If $M_0$ is a compact normally hyperbolic invariant manifold (see \Cref{sec:NH} for definition) of the unperturbed system $\dot x = f(x)$, then, for each $\e$ sufficiently close to zero, there is a unique compact invariant manifold $M_\e$ of (\ref{system.statement.fenichel}) in a neighbourhood of $M$, which is furthermore normally hyperbolic and diffeomorphic to $M_0$. This is usually expressed by saying that normally hyperbolic invariant manifolds ``persist'' under regular perturbations. Henceforth, this result will referred to as the {\it Regular Persistence Theorem}, for convenience.

In contrast, weak hiperbolicity occurs, for instance, if we are faced with
\begin{equation} \label{system.statement.fenichelweak}
	\dot x = \e f(x) + \e^2 h(x,\e).
\end{equation}
In that case, we cannot immediately identify a compact normally hyperbolic invariant manifold for $\e=0$, because the system is thus reduced to $\dot x =0$. For each $\e \neq0$, the ``unperturbed'' system $\dot x = \e f(x)$ has such a manifold, but it degenerates and loses hyperbolicity as $\e$ tends to zero, being deemed ``weakly hyperbolic". It is evident that (\ref{system.statement.fenichelweak}) can be transformed into (\ref{system.statement.fenichel}) by simply re-scaling time, which then allows the Regular Persistence Theorem to be applied, but the situation can be considerably more complicated if multiple time-scales are introduced.

A natural instance of this problematic multiple time-scales phenomenon occurs when the perturbative terms are time-dependent, i.e., 
\begin{equation} \label{system.statement.timedep}
	\dot x = \e f(x) + \e^2 h(t,x,\e).
\end{equation}
Time re-scaling transforms (\ref{system.statement.timedep}) into 
\begin{equation} \label{system.statement.timedep.rescaled}
	\dot x = f(x) + \e h\left(\frac{t}{\e},x,\e\right),
\end{equation}
a system to which classical persistence theory cannot be applied, because the perturbation term becomes singular at $\e=0$. Still, one may sensibly inquire whether persistence of invariant manifolds can still be established for \eqref{system.statement.timedep.rescaled}, which would in turn guarantee invariant manifolds for \eqref{system.statement.timedep}. This is essentially the problem we will analyze in this paper, under the provision that $h$ is time-periodic. 

The introduction of non-autonomous perturbations has another consequence worth highlighting: no more are we justified in looking for ``static'' invariant manifolds of (\ref{system.statement.timedep})---they must generally be time-dependent. To visualize this object, we will frequently introduce time as an extra variable, studying instead invariant manifolds of the corresponding autonomous system
\begin{equation} \label{system.statement.timedep.ext}
	\begin{aligned}
		&\dot s = 1, \\
		&\dot x = \e f(x) + \e^2 h\left(s,x,\e\right).
	\end{aligned}
\end{equation}
The phase space of this extended differential system will be referred to as the ``extended phase space'' of (\ref{system.statement.timedep}), for convenience. 

Periodicity in time also allows us to see $s$ as an angular variable, i.e., to consider (\ref{system.statement.timedep.ext}) to be defined in $\mathbb{S}^1 \times U$. With that in mind, it is easy to see that ``time-periodic'' invariant manifolds of (\ref{system.statement.timedep}) correspond to compact invariant manifolds of (\ref{system.statement.timedep.ext}). We will always take into account this periodic extended phase space when discussing persistence of invariant manifolds. 

\subsection{Persistence of weakly normally hyperbolic $d$-tori} \label{sec:intro.wnhim}
The main result of this paper concerns the solution of the problem described above in the case of $d$-tori, generalizing the findings of \cite{MR1740943, PNC23} for limit cycles in the plane. More precisely, it provides sufficient conditions for the persistence of attracting weakly normally hyperbolic invariant $d$-tori under small time-periodic perturbations. Its precise statement is as follows.

\begin{mtheorem} \label{thm.main.wnhim}
	Let $\e_0>0$ and $\ell,n,r,r_H,d \in \mathbb{N}^*$ be such that $r\geq 3$, $r_H \leq r$, and $d <n$. Consider the system
	\begin{equation} \label{eq:thm.main1}
		\dot x = \e^\ell f(x) + \e^{\ell+1} h(t,x,\e), \quad (t,x,\e) \in \R \times U \times (-\e_0,\e_0),
	\end{equation}
	where $f$ and $h$ are of class $C^{r}$ and $U \subset \R^n$ is an open bounded set. Assume that the following hypotheses are valid:
	\begin{enumerate}[label*=$(H.\arabic*)$]
		\item \label{hyp.1-thm1} There is $T>0$ such that $h(t+T,x,\e) = h(t,x,\e)$ for all $(t,x,\e)\in \R \times U \times (-\e_0,\e_0)$;
		\item \label{hyp.2-thm1} $\dot z = f(z)$ admits an attracting $r_H$-normally hyperbolic invariant $d$-torus $\mathcal{T}^d_0$ of class $C^3$ with trivial normal bundle.
	\end{enumerate}
	Then, there is $\e_* \in (0,\e_0)$ such that, for each $\e \in (0,\e_*)$, the extended system 
	\begin{equation} \label{eq:thm.main2}
		\begin{aligned}
		&\dot \tau = 1, \quad
		&\dot x = \e^\ell f(x) + \e^{\ell+1} h(\tau,x,\e), \quad (\tau,x) \in \R / ( \mathbb{Z} T) \times U,
		\end{aligned}
	\end{equation}
	admits an attracting $r_H$-normally hyperbolic invariant $(d+1)$-torus $\mathcal{T}^{d+1}_\e$ with trivial normal bundle. The following also hold:  
	\begin{enumerate} [label*=$(\roman*)$]
		\item \label{item.thm.main.wnhim.smooth}Each $\mathcal{T}^{d+1}_\e$ is of class $C^{r_H}$;
		\item \label{item.thm.main.wnhim.F}For each $\e \in (0,\e_*)$, there is $\mathcal{F}_\e: \R \times \R^d \to \cc^1 \times \R^n$ of class $C^1$ such that $$\mathcal{T}^{d+1}_\e = \{(\tau \mod T,\mathcal{F}_\e(\tau,\theta)): (\tau,\theta) \in \R \times \R^d \};$$
		\item \label{item.thm.main.wnhim.periodic}Each $\mathcal{F}_\e$ is $T$-periodic in $\tau$ and $1$-periodic in each entry of $\theta$;
		\item \label{item.thm.main.wnhim.Fconvergence}$\mathcal{F}_\e$ converges in the $C^1$ norm to $(\tau,\theta) \mapsto A(\theta)$ as $\e \to 0^+$, where $A:\R^d \to \R^n$ is a parametrization of $\mathcal{T}_0^d$;
		\item \label{item.thm.main.wnhim.Tconvergence}$\mathcal{T}^{d+1}_\e$ converges in the $C^1$ topology to $\R / ( \mathbb{Z} T) \times \mathcal{T}_0^d$ as $\e \to 0^+$.
	\end{enumerate}
\end{mtheorem}

A few remarks are warranted at this point. Although the concept of $r_H$-normal hyperbolicity is formally defined only in \Cref{sec:NH}, it intuitively dictates that the tangential dynamics on the invariant manifold are dominated (with degree $r_H$) by the normal dynamics. Throughout this paper, unless qualified otherwise, this term refers to what is usually named absolute normal hyperbolicity in the literature. Conversely, its relative counterpart---necessary for a continuation result in \Cref{sec:Continuation.General.Result} that applies to both concepts---will always be explicitly qualified.

Additionally, the usual notion of hyperbolicity of a limit cycle---which is assumed in the results found in \cite{MR1740943, PNC23} that are generalized by \Cref{thm.main.wnhim}---coincides with its $r_H$-normal hyperbolicity for any degree. Still in regard to the hypotheses, we highlight the following remarks, as they will be useful later.
\begin{remark} \label{rm:codimension1}
	Every $d$-torus embedded in $\R^{d+1}$ has trivial normal bundle due to orientability---see, e.g., \cite[Section 2, Exercise 18]{guillemin1974differential}. Hence, this part of \ref{hyp.2-thm1} need not be checked if $n=d+1$.
\end{remark}
\begin{remark}\label{rm:curves}
	 Normal bundles of \textbf{closed curves} (1-tori) embedded in $\R^n$, $n \geq 2$, are also trivial---see \cite[\textit{Theorem 7.2 (a)}]{kosinski1993differential} for the case $n\geq3$ and the Remark above for $n=2$.
\end{remark}

Finally, \Cref{thm.main.wnhim} can also be seen as a tool to reduce the problem of detecting invariant $(d+1)$-tori for systems of the form \eqref{eq:thm.main2} to the detection of $d$-tori for the simpler $\dot z = f(z)$. This tool can be made much more general if combined with the so-called method of averaging, as will be explained in \Cref{sec:IntroAveraging}. It can in particular be used to detect invariant $d$-tori in autonomous systems of differential equations, as demonstrated in \Cref{sec:IntroNHIT}.

\subsection{An averaging principle for invariant $d$-tori} \label{sec:IntroAveraging}
As mentioned, the original motivation for this paper is the analytical detection of attracting sets. Specifically, we set out to investigate the application of the method of averaging to locate invariant $d$-tori of non-autonomous differential equations. This section details this problem and establishes its connection with the weakly normally hyperbolic invariant tori discussed in the preceding section.

The roots of the method of averaging (also known as averaging theory) lie in 18th-century celestial mechanics, with the development of tools to study perturbations of the two-body problem by Clairaut, Lagrange and Laplace (for more details, the reader is referred to \cite[Appendix A]{verhulstsanders}). The objective was to isolate the ``secular'' terms of the perturbation, i.e., those related to long-term changes, from those causing only rapid fluctuations. If the perturbation terms are periodic, which is a natural requirement in celestial mechanics, it turns out that the first order secular terms correspond to the average of the perturbation over the period.  

In the following centuries, further development of the theory built upon those original methods to create a powerful set of tools to analyze non-autonomous differential systems (see, for instance, \cite{BM,Guckenheimer1983,verhulstsanders}), most prominently providing asymptotic approximations to their solutions via autonomous systems obtained by averaging the original systems over time. More germane to the specific context of this paper is the fact that averaging theory also allows us to obtain qualitative information about the behavior of non-autonomous systems, such as the existence of periodic solutions and integral manifolds. We present an overview of the relevant methods in what follows.

We start with a family of differential equations in the so-called \textit{standard form} for the application of averaging:
\begin{equation} \label{system.standardform.intro}
	\dot x = \sum_{i=1}^N \e^i F_i(t,x) + \e^{N+1} \tilde{F}_{N+1}(t,x,\e), \quad (t,x,\e) \in \R \times D \times (-\e_0,\e_0).
\end{equation}
where $N \in \mathbb{N}^*$, $\e_0>0$, and $D \subset \R^n$ is a bounded open set. Each function appearing on the right-hand side of (\ref{system.standardform.intro}) is assumed to be of class $C^r$ and $T$-periodic in time, with $r>0$ and $T>0$. The parameter $\e$ is perturbative, and thus usually seen as small. \Cref{system.standardform.intro} can thus be seen as a asymptotic expansion in $\e$ with coefficients depending on $(t,x)$.

A fundamental lemma of the theory of averaging (see, for instance, \cite[Lemma 2.9.1]{verhulstsanders}) guarantees the existence of a $T$-periodic change of variables 
\begin{equation} \label{eq:AvgCofV}
	x = y + \sum_{i=1}^N \e^i u_i(t,y)
\end{equation}
that transforms (\ref{system.standardform.intro}) into
\begin{equation}\label{system.standardform.intro.avg}
	\dot y = \sum_{i=1}^N \e^i g_i(y) + \e^{N+1} R_{N+1} (t,y,\e).
\end{equation}
Effectively, time dependence is pushed to $\mathcal{O}(\e^{N+1})$ with this transformation, and each new time-independent coefficient $g_i$ is named the \textit{averaged function} of $i$-th order. It is known (see, e.g., \cite{Novaes21b}) that $g_i$ is of class $C^{r-i+1}$ and that $R_{N+1}$ is of class $C^{r-N}$. Also, justifying the name given to the method, the first averaged function is the time average of $F_1$ over its period, i.e.,
\begin{equation*}
	g_1 (y) = \frac{1}{T} \int_0^T F_1(t,y)\, dt.
\end{equation*}

The established procedure is then to study the first non-trivial order of (\ref{system.standardform.intro.avg}) as an approximation of the non-autonomous system. Accordingly, let $\ell \in \{1,\ldots, N\}$ be the first index $i$ for which $g_i$ does not vanish identically. We will refer to $\dot z = g_\ell(z)$ as the \textit{guiding system} of (\ref{system.standardform.intro}). A classic result (see, for instance, \cite[Theorem 4.1.1]{Guckenheimer1983} or \cite{LNT14}) guarantees that each \textit{simple equilibrium point} of the guiding system corresponds to a \textit{periodic solution} of (\ref{system.standardform.intro}). This result appeared already, albeit in a restricted form, in the works of Krylov and Bogoliubov in the 1930's (see \cite{Kryloff1935}), but was further generalized in many directions by different authors. 

It acquires an elegant geometric significance if, as in \Cref{sec:intro.wnhim}, we consider the extended phase space of (\ref{system.standardform.intro}), i.e., the phase space $\mathbb{S}^1 \times D$ of
\begin{equation} \label{system.standardform.intro.ext}
	\begin{aligned}
		&\dot \tau =1, \quad
		& \dot x = \sum_{i=1}^N \e^i F_i(\tau,x) + \e^{N+1} \tilde{F}_{N+1}(\tau,x,\e).
	\end{aligned}
\end{equation}
In this setting, simple equilibria of the guiding system correspond to \textit{limit cycles} of (\ref{system.standardform.intro.ext}). This correspondence becomes even more intriguing if we take into account that, under certain conditions, \textit{limit cycles} of the guiding system correspond to \textit{invariant tori} in the extended phase space.
 
The original proof of this result---which was restricted to first order averaging only---seems to be due to Bogoliubov and Mitropolsky \cite{BM}. A generalization of the class of systems encompassed by this initial theorem was then proved by Hale \cite{hale61}, still restricted to first order averaging. More recently, it has been extended to include higher order averaging, as well as considerations of normal hyperbolicity of the resulting invariant tori in some restricted cases (see \cite{NP24,PNC23}). 

The question that sparked the development of this paper is thus if this correspondence can be generally extended to higher dimensional invariant tori. More precisely, provided that we find an invariant $d$-torus in the phase space of the guiding system of \eqref{system.standardform.intro}, are there sufficient conditions under which a corresponding invariant $(d+1)$-torus in the extended phase space of this system can be guaranteed to exist? It is easy to see, by comparing (\ref{system.standardform.intro.avg}) with (\ref{eq:thm.main1}), that this question is closely related to the already discussed persistence of weakly normally hyperbolic tori. Consequently, a corollary of \Cref{thm.main.wnhim} is the following affirmative answer.

\begin{mtheorem} \label{thm.main.avg}
	Assume $r-\ell \geq 3$ and let $d,r_H \in \mathbb{N}^*$ be such that $r_H \leq r- \ell$. If the guiding system $\dot z = g_\ell (z)$ of (\ref{system.standardform.intro}) has an attracting $r_H$-normally hyperbolic invariant $d$-torus $\mathcal{T}^d_0$ at least of class $C^{3}$ and with trivial normal bundle, then, for sufficiently small values of $\e>0$, (\ref{system.standardform.intro}) has an attracting $r_H$-normally hyperbolic invariant $(d+1)$-torus $\mathcal{T}^{d+1}_\e$ with trivial normal bundle in its extended phase space, which is, in particular, of class $C^{r_H}$. Moreover, $\mathcal{T}^{d+1}_\e$ converges to $\cc^1 \times \mathcal{T}^d_0$ as $\e \to 0^+$.
\end{mtheorem}

\subsection{Ocurrence of normally hyperbolic invariant $d$-tori in polynomial systems} \label{sec:IntroNHIT}

The final topic explored in this paper concerns the number of normally hyperbolic invariant tori in polynomial differential systems. Besides being an interesting problem in itself, it demonstrates that the aforementioned results are applicable to settings that, at first glance, appear entirely unrelated to the perturbative non-autonomous systems introduced thus far.

Let ${\bf P} = (P_1,\ldots, P_{d+1}):\R^{d+1} \to \R^{d+1}$ be such that each $P_i$ is a polynomial function, and define $\deg({\bf P}) = \max \{ \deg(P_i): i =1,\ldots,d+1\}$. Denoting by $\tau_{r}^d({\bf P})$ the number of $r$-normally hyperbolic invariant $d$-tori of the differential system ${\bf \dot x} = {\bf P}({\bf x})$, we define, for each $m \in \N$,
\begin{equation*}	N_r^d(m) : = \sup\{\tau_r^d({\bf P}): \deg({\bf P}) \leq m\}.
\end{equation*}
If finite, $N^d_r(m)$ represents the maximum number of $r$-normally hyperbolic codimension-$1$ invariant tori that can appear in the phase space of a $(d+1)$-dimensional polynomial differential system of degree $m$. We restrict our attention to codimension-$1$ invariant tori, since, as we will show, the corresponding number is unbounded for invariant tori of higher codimension. Of course, since every $r$-normally hyperbolic invariant manifold is, in particular, $r'$-normally hyperbolic for $r'\leq r$, it follows that 
\begin{equation} \label{eq:relationNrNr'}
	r ' \leq r \implies N^d_{r'}(m) \geq N_r^d (m).
\end{equation}

Investigating whether $N_r^d(m)$ is finite can be seen as an extension of the second part of Hilbert's 16th problem (see, e.g., \cite{Ilyashenko02}), which is concerned with the maximum number $H(m)$ of limit cycles of polynomial differential systems of degree $m$ in the plane. Since each limit cycle is in particular an invariant 1-torus, and considering that the maximum number of limit cycles has been shown to be achievable using only hyperbolic limit cycles \cite{gasull2025note}, it is easy to see that 
\begin{equation} \label{eq:ineqN1H}
	N_r^1(m) \geq H(m),
\end{equation}
for any $r \geq 1$. Hence, any upper bound of $N_r^1(m)$ is also an upper bound of $H(m)$.

As a contribution to this problem, we prove that, if $r \geq 3$, the following relation between the numbers $N_r^d$ for different values of $d$ can be established: \begin{equation}\label{Nhd}
	N_r^d(m) \geq N_r^{d-1} \left(\left\lfloor \frac{m}{2} \right\rfloor - 1\right).
\end{equation}
We also investigate the asymptotic growth of $N^{d+1}_r(m)$ with respect to $m$, proving it has to grow at least as fast as $m^{d+1}$. More precisely, we show that
\begin{equation}\label{liminf}
	\liminf_{m \to + \infty} \frac{N_r^d(m)}{m^{d+1}} \geq  \frac{4}{(2^{d+1} +d - 2)^{d+1}}.
\end{equation}
Those results generalize the findings of \cite{NP25}, which were mostly restricted to 3-dimensional systems. 

Estimates \eqref{Nhd} and \eqref{liminf} are corollaries of \Cref{thm:app}, itself a direct consequence of \Cref{thm.main.avg}. It provides a mechanism to construct a family of $(d+1)$-dimensional differential systems with a prescribed number of $r$-normally hyperbolic attracting invariant $d$-tori from a given $d$-dimensional differential system possessing at least the same number of $r$-normally hyperbolic attracting invariant $(d-1)$-tori. The key feature of this construction, which allows \eqref{Nhd} and \eqref{liminf} to be deduced, is that it preserves polynomiality: the resulting family of $(d+1)$-dimensional differential system is polynomial of degree $2q+2$ whenever the original $d$-dimensional system is polynomial of degree $q \in \N$. 

\begin{mtheorem}\label{thm:app}
	Let $1 \leq d \leq n-1$ and $r_H$ be positive integers, and $P:\R \times \R^{n-1} \to\R$ and $Q:\R \times \R^{n-1} \to\R^{n-1}$ be smooth functions. Assume that the $n$-dimensional differential system
	\begin{equation}\label{eq:2dpde}
		\begin{cases}
			\dot u=P(u,V),\\
			\dot V=Q(u,V),
		\end{cases}
	\end{equation}
 	has an attracting $r$-normally hyperbolic $d$-tori $\mathcal{T}^d_0$ of class $C^3$ with trivial normal bundle, contained in a bounded region $$K=(0,b)\times U \subset \R \times \R^{n-1},$$ and consider the following one parameter family of $(n+1)$-dimensional differential systems:
	\begin{equation}\label{eq:3dpde}
		\begin{cases}
			\dot x=-y,\\
			\dot y=x+\e\,y\, P(x^2+y^2,Z),\\
			\dot Z=2\,\e\,y^2 Q(x^2+y^2,Z),
		\end{cases} \qquad (x,y,Z)\in\R\times\R\times\R^{n-1}.
	\end{equation}
Then, there exists $\bar \e>0$ such that the differential system \eqref{eq:3dpde} has an attracting $r$-normally hyperbolic invariant $(d+1)$-torus $\mathcal{T}^{d+1}_{\e}$ with trivial normal bundle for every $\e\in(0,\bar\e)$. Moreover, $\mathcal{T}^{d+1}_{\e}$ converges to $\mathbb{S}^1\times \mathcal{T}^{d}_{0}$ as $\e$ goes to $0$.
\end{mtheorem}

\subsection{Structure of the paper}
The paper is organized as follows. \Cref{sec:NH} establishes the necessary terminology and recalls standard properties of normally hyperbolic invariant manifolds. \Cref{sec:Continuation.General.Result} contains the proof of a general technical result concerning continuation of weakly normally hyperbolic manifolds, while \Cref{sec:existingresultshenry} reviews an important theorem on the existence of invariant manifolds in a class of singularly perturbed systems. The conclusions of these two sections are then employed in \Cref{sec:proofthmA} to prove \Cref{thm.main.wnhim}. Next, \Cref{sec:proofsBC} provides the proofs of Theorems \ref{thm.main.avg} and \ref{thm:app}. Finally, \Cref{sec:ToriPoly} applies the machinery developed in the preceding sections to investigate normally hyperbolic invariant $d$-tori in polynomial systems.

\section{Relative and absolute normal hyperbolicity} \label{sec:NH}
Until now, we have mentioned normal hyperbolicity without formally defining it. Since the concept frequently appears with different meanings in the literature - Hirsch, Pugh and Shub, which apparently originated the term, provided no less than four definitions (see \cite{hirschpughshub}) - we find it appropriate to explicitly state the notions of normal hyperbolicity that we will adopt. For convenience, we will only define normal hyperbolicity for attracting manifolds, as those are the only ones effectively considered in this paper.

Let $n,r \in \mathbb{N}^*$ and $\phi_t$ be the flow of the $C^r$ differential system $\dot x = F(x)$ in $\R^n$ having a compact connected invariant manifold $M$ of class $C^1$. If $N$ is the normal bundle with respect to the standard metric on $\R^n$, the tangent bundle of $\R^n$ over $M$ admits a continuous splitting
\begin{equation*}
	T\R^n |_M = TM \oplus N.
\end{equation*}
Denote by $\Pi$ the orthogonal projection of $T\R^n|_M$ onto $N$. Following Fenichel (see \cite{Fenichel1971} or \cite{wigginsnormally}), for each $t \in \R$ and $p \in M$, define the operators $A_t(p): T_p M \to T_{\phi_{-t}(p)}M$ and $B_t(p): N_{\phi_{-t}(p)} \to N_p$ by 

\begin{equation} \label{eq:defABFenichel}
	\begin{aligned}
		& A_t(p) : = D\left(\phi_{-t}|_M\right) (p) \quad \text{and} \quad B_t(p) : = \Pi \cdot D\phi_{t} (\phi_{-t}(p))|_N.
	\end{aligned}
\end{equation}
They are designed to trace, respectively, variations of vectors tangent and transversal to $M$, so that a comparison can be established. 

 We begin by defining \textbf{absolute} normal hyperbolicity, which is of greater significance in the context of this paper. We will mimic the homonymous notion introduced for diffeomorphisms in \cite{hirschpughshub} (see also \cite{Eldering2013Normally}), by requiring uniform control of the rates of contraction on the manifold. More precisely, we assume that there are $K_A,K_B>1$ and $\rho_M,\rho_N>0$ such that
\begin{itemize}
	\item $\| A_t(p) \| \leq K_A e^{\,\rho_M |t|}$, for all $t \in \R$ and all $p \in M$; 
	\item $\|B_t(p)\| \leq K_B e^{-\rho_N t }$, for all $t\geq0$ and all $p \in M$.
\end{itemize}
We then ask that the contraction along the normal direction is $r$-times stronger than the tangential one, i.e.,
\begin{definition} \label{def:AbsNH}
	Let $r \geq1$. $M$ is \textbf{absolutely} $r$-normally hyperbolic if $\rho_N > r \rho_M$.
\end{definition}
\noindent Incidentally, we remark that the requirement of compactness of $M$ can be relaxed if one introduces uniform estimates in its place \cite{Eldering2013Normally}.

In contrast with \Cref{def:AbsNH}, \textbf{relative} normal hyperbolicity---adopted, for example, by Fenichel in \cite{Fenichel1971} (see also \cite{hirschpughshub})---requires domination of the normal contraction over its tangential counterpart \textit{for each individual choice of orbit} on the manifold $M$.  This can be achieved in terms of the so-called generalized Lyapunov type numbers (see \cite{Fenichel1971,wigginsnormally}): 
\begin{equation*} \label{eq:defLyapunovtypenumbers}
	\begin{aligned}
		&\nu(p) : = \inf \left\{a\geq0: \frac{\| B_t(p)\| }{a^t} \to 0 \; \text{as} \; t \to \infty\right\}, \\[2pt]
		&\sigma(p) : = \inf \left\{s\geq0: \|A_t(p)\| \, \| B_t(p)\|^s \to 0 \; \text{as} \; t \to \infty\right\},
	\end{aligned}
\end{equation*}
where $p \in M$.
\begin{definition} \label{def:RelNH}
	$M$ is \textbf{relatively} $r$-normally hyperbolic if $\nu(p)<1$ and $\sigma(p) < 1/r$ for all $p \in M$.
\end{definition}

As in the absolute case, the number $\nu(p)$ is used only to ensure that normal vectors contract. However, $\sigma(p)$ provides a comparison between normal and tangential contraction rates for the orbit starting at $p$. Hence, relative normal hyperbolicity may apply even if the normal contraction rate along one orbit does not dominate the tangential rate along another. By contrast, the absolute notion certainly implies the relative one.

A few remarks are in order concerning \Cref{def:AbsNH,def:RelNH}. The first is that these definitions are independent of the metric in $\R^n$, provided that the splitting and the projection are also taken with respect to chosen metric (see, e.g., \cite{Fenichel1971,hirschpughshub}). Furthermore, both properties are invariant under coordinate transformation and time reparametrization (see, e.g., \cite{Arakaki2026}). Finally, since the vector field $F$ is of class $C^r$, if the $C^1$ manifold $M$ is (absolutely or relatively) $r$-normally hyperbolic, it is known that $M$ must be \textit{a fortiori} of class $C^r$ (see, e.g., \cite[Theorem 4.1]{hirschpughshub}).

The weaker notion of relative normal hyperbolicity is crucial due to its sufficiency for persistence of the invariant manifold under $C^1$ perturbations of the vector field, i.e.,
\begin{theorem}[\cite{Fenichel1971}] \label{thm:fenichelper}
	If $M$ is relatively $r$-normally hyperbolic, then, for any $C^r$ vector field $G$ that is sufficiently $C^1$-close to $F$, there is a compact manifold $M_G$ of class $C^r$ that is invariant under $G$, $C^r$-diffeomorphic to $M$, and converges $C^1$ to $M$ as $G \to F$.
\end{theorem}
\noindent In \cite{hirschpughshub}, an equivalent result is proved for a manifold invariant under the action of a diffeomorphism, in which case relative normal hyperbolicity is proved to be not only sufficient, but also necessary for persistence \cite{Ma1978}.

Due to its prevalence in the results proved in this paper, we will generally refer to absolute normal hyperbolicity simply as normal hyperbolicity, omitting further qualification. We chose to introduce relative normal hyperbolicity not only as a means of contrasting the absolute notion, but also because \Cref{thm:GC}, a technical result that formalizes the continuation method introduced in \cite{Kopell85} to a broad class of systems, encompasses both notions.

\subsection{Verifying normal hyperbolicity}
Verifying absolute or relative normal hyperbolicity can prove itself challenging if one works directly with their definition. In this section, we present alternative definitions, which will be helpful in \Cref{sec:Continuation.General.Result}, as well as a classic result concerning an uniform bound for the generalized Lyapunov type numbers, known as Fenichel's Uniformity Lemma. More importantly, by analogy with the relative case, we define a third generalized Lyapunov type number to facilitate the verification of absolute normal hyperbolicity.

We begin with the following lemma, concerning the calculation of the numbers $\nu$ and $\sigma$ appearing in the definition of relative normal hyperbolicity. It is a classic result in the theory of normally hyperbolic invariant manifolds (see \cite{wigginsnormally}), and thus will be stated without proof. 
\begin{lemma}\label{lemma.lypaunovnumberseqdef}
	The generalized Lyapunov type numbers $\nu$ and $\sigma$ can be equivalently defined by
	\begin{equation*}
		\nu(p) = \limsup_{t\to +\infty} \|B_t(p)\|^{\frac{1}{t}}, \quad \sigma(p) = \limsup_{t\to +\infty} \frac{\ln \|A_t(p)\|}{-\ln\|B_t(p)\|}.
	\end{equation*}
\end{lemma}
As seen above, relative normal hyperbolicity is verified by assigning certain bounds for those quantities. A classic result concerning such bounds is the following Uniformity Lemma, whose proof can be found in \cite{Fenichel1971}.
\begin{lemma}[Uniformity Lemma]\label{lemma.uniformitylemma}
	If $a \in (0,1]$ and $s>0$ are such that $\nu(p) < a$ and $\sigma(p) < s$ for all $p \in M$, then there are $\hat{a}\in (0,a)$ and $\hat{s} \in (0,s)$ such that $\nu(p) <\hat{a}$ and $\sigma(p) < \hat{s}$ for all $p \in M$.
\end{lemma}

Finally, in order to more easily verify absolute normal hyperbolicity, we define a third generalized Lyapunov type number:
\begin{equation*}
	\xi_{\pm}(p) : = \inf \left\{a\geq 0 : \frac{\|A_t(p)\|}{a^{|t|}} \to 0 \; \text{as} \; t \to \pm \infty \right\}.
\end{equation*}The argument used in \cite{wigginsnormally} to prove \Cref{lemma.lypaunovnumberseqdef} can easily be adapted to guarantee that 
\begin{equation*}
	\xi_{\pm}(p) = \limsup_{t \to \pm \infty} \|A_t(p)\|^{\frac{1}{|t|}}.
\end{equation*}With that in mind, we define the absolute Lyapunov type numbers 
\begin{equation*}
	\nu(M) : = \sup_{p \in M} \nu(p) \quad \text{and} \quad \xi_{\pm}(M) = \sup_{p \in M} \xi_{\pm}(p).
\end{equation*}
We are now able to prove a result that aids in verifying absolute normal hyperbolicity.
\begin{lemma} \label{lemma.verifyabsNH}
	If $M$ is a compact manifold of class $C^1$ invariant under the flow of $\dot x = F(x)$ for which $\nu(M)<1$ and there is $r\geq 1$ such that $ \nu (M) \cdot \left(\xi_{\pm}(M)\right)^r<1$, then $M$ is absolutely $r$-normally hyperbolic.
\end{lemma}
\begin{proof}
	Let $a > \xi_{\pm}(M)$. Then, by definition of $A_t(p)$ and $\xi_{\pm}$, it follows that there is $K_1>0$ such that 
	\begin{equation*}
		\|A_t(p)\| \leq K_1 a^{|t|} = K_1 e^{\ln a \, |t|}.
	\end{equation*}	for all $t \in \R$ and all $p \in M$. Hence, it is clear that the tangential requirement for absolute normal hyperbolicity is satisfied for some $\rho_M\leq \ln a$.
	Likewise, if $b< \nu(M)<1$, there is $K_2>0$ such that 
	\begin{equation*}
		\|B_t(p)\| \leq K_2 b^{t} = K_2 e^{- (-\ln b) \, t}.
	\end{equation*}	for all $t\geq 0$ and all $p \in M$. Thus, the requirement for absolute normal hyperbolicity concerning the normal direction is satisfied for some $\rho_N \geq - \ln b$.
	
	Since $ \nu (M) \cdot \left(\xi_{\pm}(M)\right)^r<1$, it is easy to see that $a$ and $b$ can be chosen such that $b \cdot a^r <1$. In that case, by taking the logarithm on both sides, it follows that $\ln b + r \ln a <0$, that is, $$\rho_N \geq -\ln b > r \ln a \geq r \rho_M,$$
	as wished.
\end{proof}

\section{The continuation method: a general result} \label{sec:Continuation.General.Result}
In \cite{Kopell85}, a method to prove a result similar to \Cref{thm.main.wnhim} was proposed. This method was later given a more general treatment in \cite[Section 7.3]{wigginsnormally}. However, a gap in its proof strategy was identified in \cite{MR1740943}. In what follows, we provide an overview of the method, its gap and, more importantly, how we bridge this gap in the case of $d$-tori investigated in \Cref{thm.main.wnhim}.

Let $r_H$ be a positive integer. To understand the proposed approach, we turn our attention to the one-parameter family (\ref{eq:thm.main1}), assuming that hypotheses \ref{hyp.1-thm1} holds and that $\dot z = f(z)$ has an attracting, relatively $r_H$-normally hyperbolic, $d$-dimensional, compact invariant manifold $M_0$ (not necessarily a $d$-torus).

We embed (\ref{eq:thm.main1}) into the two-parameter family
\begin{equation*} \label{eq:continuationfamily}
	\dot x = \delta^\ell f(x) + \e^{\ell+1} h(t,x,\e),
\end{equation*}
with $(\delta,\e) \in (0,\e_0) \times (0,\e_0)$. By re-scaling time to $s = \delta^\ell t$ and adding $s$ as an extra angular variable modulo the time period $T$, we obtain
\begin{equation} \label{eq:continuationfamilyResAut}
	\begin{aligned}
		&\dot s =1, \\
		&\dot x =  f(x) + \frac{\e^{\ell+1}}{\delta^\ell} \, h\left(\frac{s}{\delta^\ell},x,\e\right).		
	\end{aligned}
\end{equation}
By fixing $\delta$ and restricting $\e \leq \delta$, \cref{eq:continuationfamilyResAut} becomes a \textit{regular} $\mathcal{O}(\e)$-perturbation of
\begin{equation} \label{eq:continuationfamilyResAutUnp}
	\begin{aligned}
		&\dot s =1,\\
		&\dot x =  f(x),
	\end{aligned}
\end{equation} 
which admits an attracting, relatively $r_H$-normally hyperbolic manifold $\cc^1 \times M_0$. Then, \Cref{thm:fenichelper} guarantees that, for each $\delta\in (0,\e_0)$, there is $\e_1(\delta) \in (0,\e_0)$ such that, if $\e \leq \e_1(\delta)$, \cref{eq:continuationfamilyResAut} admits an invariant manifold $\mathcal{M}_{\delta,\e}$, which is $C^1$-near $\cc^1 \times M_0$.  

The idea of the proposed method of proof is to show, via direct estimation of the generalized Lyapunov type numbers $\nu^{\delta,\e}$ and $\sigma^{\delta,\e}$ (see \Cref{sec:NH}) of the perturbed manifold, that, if $\delta$ is chosen sufficiently small, then $\nu^{\delta,\e}(p)<1$ and $\sigma^{\delta,\e}(p)<1/r_H$ for any $p \in \mathcal{M}_{\delta,\e}$, even as $\e$ approaches $\delta$.  This in turn, would guarantee that $\mathcal{M}_{\delta,\delta}$ is relatively $r_H$-normally hyperbolic, which amounts to the existence of a relatively $r_H$-normally hyperbolic invariant manifold in the extended phase space of \cref{eq:continuationfamily} for $\e=\delta$ or, equivalently, of \cref{eq:thm.main1}.

The gap in the proof is not in the specific manner in which the generalized Lyapunov type numbers are bounded, but in the incorrect implicit assumption that the family of ``smoothly deforming manifolds'' parameterized by $\e$ can only lose relative $r_H$-normal hyperbolicity at some $\e_*$ by having $\sup \{\nu^{\delta,\e}(p): p \in \mathcal{M}_{\delta,\e}\} \to 1$ or $\sup \{\sigma^{\delta,\e}(p): p \in \mathcal{M}_{\delta,\e}\} \to 1/r_H$ as $\e \to \e_*$. 

In fact, what can occur is that those numbers are bounded away from those thresholds as $\e$ nears $\e^*$, yet the family of manifolds does not converge to a $C^1$ manifold. An illustrative example---provided by Chicone and Liu in \cite{MR1740943}---is that of hyperbolic limit cycles approaching a homoclinic loop at a hyperbolic saddle with distinct eigenvalues on the plane. In that case, loss of normal hyperbolicity occurs by simple loss of smoothness of the ``deforming manifold'' as described, the behavior of the generalized Lyapunov type numbers notwithstanding.

In \cite{MR1740943}, Chicone and Liu corrected the method above as follows: for each $\delta \in (0,\e_0)$, let $A^\delta \subset [0,\delta]$ be the maximal interval with left endpoint at $0$ for which, if $\e \in A^\delta$, then \cref{eq:continuationfamilyResAut} admits a relatively $r_H$-normally hyperbolic invariant manifold $\mathcal{M}_{\delta,\e}$ as described above. $A^\delta$ is non-empty because $0 \in A^\delta$. It is also open due to persistence of relatively normally hyperbolic manifolds under regular perturbations (see \Cref{thm:fenichelper}). Thus, we can conclude that $A^\delta = [0,\delta]$ provided we show that it is also closed. To do so, they suggest we define $\e_*(\delta) = \sup A^\delta$ and attempt to complete two steps: show that $\mathcal{M}_{\delta,\e_*(\delta)}$ exists as a $C^1$ invariant manifold---for instance, by realizing it as the graph of a function which is the limit of an equicontinuous family; and verify the definition of relative $r_H$-normal hyperbolicity of $\mathcal{M}_{\delta,\e_*(\delta)}$ directly. If that is done, then $A^\delta$ is indeed closed, so that $\e_*(\delta) = \delta$.

Here, we suggest a slightly different approach, which works when existence of the persisting family of $C^1$ invariant manifolds $\mathcal{M}_{\delta,\e}$ extending up to $\e=\delta$ can be asserted \textit{a priori} for sufficiently small $\delta>0$. If we can do so, we are justified in calculating the generalized Lyapunov type numbers of those manifolds. Thus, estimation of those quantities alone suffices to guarantee that every $\mathcal{M}_{\delta,\e}$, including when $\e =\delta$, is relatively $r_H$-normally hyperbolic. This approach leads to \Cref{thm:GC} below, which also addresses the absolute case. Its employment in the case of tori is the main subject of this paper and will be discussed in detail in \Cref{sec:proofthmA}.

\begin{theorem} \label{thm:GC} Let $\ell, n$, and $r_H$ be positive integers. Consider the family
	\begin{equation} \label{eq:main-thmGC}
		\dot x = \e^\ell f(x) + \e^{\ell+1} h(t,x,\e), \quad (t,x,\e) \in \R \times U \times (0,\e_0),
	\end{equation}
with $\e_0>0$, $f$ and $h$ of class $C^{r_H}$, and satisfying the following hypotheses:
	\begin{enumerate}
		\nameditem{$(H.1)$} \label{hyp.1-thmGC} There is $T>0$ such that $h(t+T,x,\e) = h(t,x,\e)$ for all $(t,x,\e)\in \R \times U \times (0,\e_0)$;
		\nameditem{$(H.2')$} \label{hyp.2'-thmGC} $\dot z = f(z)$ has an attracting, relatively (resp. absolutely) $r_H$-normally hyperbolic, $d$-dimensional, compact invariant manifold $M_0$;
		\nameditem{$(H.P)$} \label{hyp.P-thmGC} The family (\ref{eq:continuationfamilyResAut}) associated to (\ref{eq:main-thmGC}) admits, for $\delta$ sufficiently small and $\e \leq \delta$, a compact invariant manifold $\mathcal{M}_{\delta,\e}$ of class $C^1$, given as the image $\mathcal{D}_{\delta,\e} (\cc^1 \times M_0)$ of a family of embeddings $\mathcal{D}_{\delta,\e}: \cc^1 \times M_0 \to \cc^1 \times \R^{n}$ such that $\| \mathcal{D}_{\delta,\e} - \mathcal{I}\|_{C^1} \to 0$ uniformly on $\delta$ as $\e \to 0^+$, where $\mathcal{I}: \cc^1 \times M_0 \to \cc^1 \times \R^n$ is the standard inclusion.
	\end{enumerate}
	Then, there is $\e_1 \in (0,\e_0)$ such that the extended system 
	\begin{equation*} \label{eq:thm.GC-ext}
		\begin{aligned}
			&\dot \tau = 1, \quad
			&\dot x = \e^\ell f(x) + \e^{\ell+1} h(\tau,x,\e), \quad (\tau,x) \in \R / ( \mathbb{Z} T) \times U,
		\end{aligned}
	\end{equation*} admits an attracting, relatively (resp. absolutely) $r_H$-normally hyperbolic, compact invariant manifold $\mathcal{N}_\e$ in for each $\e \in (0,\e_1)$.
\end{theorem}

\subsection{Proof of \Cref{thm:GC}}
	
	The proof will follow a main line of argument that will be occasionally interposed with lemmas tackling some of its technical aspects. In summary, we consider \cref{eq:continuationfamilyResAut} and prove that $\mathcal{M}_{\delta,\e}$ is relatively (resp. absolutely) $r_H$-normally hyperbolic if $\e\leq \delta$ and $\delta$ is made sufficiently small. Having done that, it suffices to re-scale time back to $t = s/ \delta^{\ell}$ and change variables from $s$ to $\tau$ via $\tau = s / \delta^{\ell}$ to obtain a family $\mathcal{N}_{\delta,\e}$ of invariant manifolds of
	$$
	\dot \tau = 1 , \quad \dot x = \delta^\ell f(x) + \e^{\ell+1}h(\tau,x,\e).
	$$
	One can thus define $\mathcal{N}_{\e}:= \mathcal{N}_{\e,\e}$, as it is easily verified to satisfy the stated conditions. Hence, all that is left to prove is that each $\mathcal{M}_{\delta,\e}$ is indeed relatively (resp. absolutely) $r_H$-normally hyperbolic, which is done below.

	Hypothesis \ref{hyp.P-thmGC} ensures that, if $\delta$ is chosen sufficiently small and $\e\leq \delta$, the invariant manifolds $\mathcal{M}_{\delta,\e}$ can be made arbitrarily $C^1$-near to $\mathcal{M}_0:= \cc^1 \times M_0$. In particular, for sufficiently small $\delta$, the normal bundle $N^0$ of $\mathcal{M}_0$ is still transversal to $\mathcal{M}_{\delta,\e}$. More precisely, take a tubular neighborhood $U$ of $\mathcal{M}_0$ and the bundle projection $\pi:U \to \mathcal{M}_0$ associated to $N^0$. Then, $\pi|_{\mathcal{M}_{\delta,\e}}$ is a diffeomorphism for small values of $\delta$, and we can define the vector bundle $I^{\delta,\e}$ on $\mathcal{M}_{\delta,\e}$ by setting $I^{\delta,\e}_p : = N^0_{\pi(p)}$. One can verify that $I^{\delta,\e}$ is transversal to $T\mathcal{M}_{\delta,\e}$ if $\delta$ is chosen sufficiently small and $\e \leq \delta$. 
	
	For $q \in \mathcal{M}_{\delta,\e}$, one can thus decompose any vector $v \in T_q (\cc^1 \times \R^n)$ as a unique sum of $v_T \in T_q\mathcal{M}_{\delta,\e}$ and $v_I \in I_q^{\delta,\e}$. We can then define an adapted metric $\langle\cdot,\cdot\rangle_q^{\delta,\e}$ over $\mathcal{M}_{\delta,\e}$ by setting
	\begin{equation*}
		\langle u, v \rangle_q^{\delta,\e} = \langle u_T, v_T\rangle_q + \langle u_I,v_I \rangle_q,
	\end{equation*}	for any $q \in \mathcal{M}_{\delta,\e}$ and $u,v \in T_q(\cc^1 \times \R^n)|_{\mathcal{M}_{\delta,\e}}$, where $\langle \cdot, \cdot \rangle$ denotes the standard flat metric on $\cc^1 \times \R^n$. It is easy to see that $T\mathcal{M}_{\delta,\e}$ and $I^{\delta,\e}$ are orthogonal under this metric, whose corresponding norm will be denoted by $\|\cdot\|_q^{\delta,\e}$. As remarked in \Cref{sec:NH}, the use of this metric does not change the definitions of normal hyperbolicity.
	
	Moreover, we define $\Pi_{\delta,\e} : T(\cc^1 \times \R^n)|_{\mathcal{M}_{\delta,\e}} \to I_{\delta,\e}$ to be the bundle projection with respect to the splitting $T(\cc^1 \times \R^n) = T\mathcal{M}_{\delta,\e} \oplus I_{\delta,\e}$. Analogously, $\Pi_0:T(\cc^1 \times \R^n)_{\mathcal{M}_0} \to N^0$ is the bundle projection with respect to $T(\cc^1 \times \R^n)_{\mathcal{M}_0} = T\mathcal{M}_0 \oplus N^0$.
	
	For convenience, we will work in the coordinate space $(s,x) \in \R \times \R^{n} = \R^{n+1}$ to calculate the generalized Lyapunov type numbers, which can be done because the covering map $(s,x) \mapsto (s \mod T , x)$ is a local isometry. In those coordinates and considering the identification $T\R^{n+1} = \R^{n+1}$, the fact that $\mathcal{M}_0 = \R \times M_0$ ensures that $N^0 \subset \{0\} \times \R^n$. Hence, each $I^{\delta,\e}$ has only vectors with vanishing first component. 
	
	On $\R^{n+1}$, we can define the pullback norm corresponding to $\| \cdot \|_q^{\delta,\e}$, and we shall compare it with the standard norm. Lifting the projections $\Pi_{\delta,\e}$ and $\Pi_0$, we obtain functions from $\R^{n+1}$ into the space of linear operators defined on $\R^{n+1}$. We conveniently maintain the notation henceforth applied to objects in $\cc^1 \times \R^n$ to their lifted counterparts. 
	
 	As linear subspaces of $\R^{n+1}$, \ref{hyp.P-thmGC} ensures that, for each $p \in \mathcal{M}_0$,	$$T_{\pi|_{\mathcal{M}_{\delta,\e}}^{-1}(p)}\mathcal{M}_{\delta,\e} \to T_{p} \mathcal{M}_0$$ uniformly on $\delta$ as $\e \to 0^+$, provided that $\e \leq \delta$. Hence, by letting $ q = \pi|_{\mathcal{M}_{\delta,\e}}^{-1}(p)$, choosing $\delta$ sufficiently small, and restricting $\e\leq \delta$, the projection $\Pi_{\delta,\e} (\pi|_{\mathcal{M}_{\delta,\e}}^{-1}(p))$ can be made arbitrarily close to $\Pi_0(p)$. Since $p$ is in the compact manifold $\mathcal{M}_0$, we conclude that there is a non-decreasing function $\eta(\delta)>0$ such that $\eta(\delta) \to 0$ as $\delta \to 0^+$ and \begin{equation*}
 		\|(\Pi_{\delta,\e}(q) - \Pi_0(p)) \cdot v \| \leq \eta(\delta) \cdot \|v\|,
 	\end{equation*}
	 for $p \in \mathcal{M}_0$, $v \in \R^{n+1}$, and $\e\leq \delta$, where $\|\cdot\|$ denotes the usual norm in $\R^{n+1}$.
 
 	\begin{lemma}
 	  	Let $\delta_0>0$ be small enough to allow the construction above to be performed for $0\leq \e \leq \delta \leq \delta_0$ and $\delta \neq 0$. In that case, the norms $\|\cdot\|^{\delta,\e}$ and $\|\cdot\|$ satisfy
 		\begin{equation} \label{eq:thmGC-NormEq}
 			\frac{1}{\left(1 + 2 \eta (\delta) \right)^2} \|v\|^2 \leq \left(\| v \|_q^{\delta,\e}\right)^2 \leq \left(1 + 2 \eta (\delta) \right)^2 \|v\|^2,
 		\end{equation}
 		for all $ q \in \mathcal{M}_{\delta,\e}$ and all $v \in \R^{n+1}$
 	\end{lemma}
 
\begin{proof}
	By definition of $\| \cdot \|_q^{\delta,\e}$, the norm of any $v \in \R^{n+1}$ satisfies
	\begin{equation*}
		\begin{aligned}
			\left(\| v \|_q^{\delta,\e}\right)^2 &= \|\Pi_{\delta,\e}(q) \cdot v\|^2 + \|(I-\Pi_{\delta,\e}(q))\cdot  v\|^2 \\& \leq \|\Pi_0(p)\cdot v\|^2 + \|(I-\Pi_0(p)) \cdot v\|^2 + ( 4\eta(\delta) + 2 \eta^2(\delta) ) \|v\|^2.
		\end{aligned}
	\end{equation*}	Considering that the splitting associated to $\Pi_0(p)$ is orthogonal with respect to $\|\cdot \|$, it follows that $\|v\|^2 = \|\Pi_0(p) \cdot v\|^2 + \|(I-\Pi_0(p)) \cdot v\|^2$. Hence, 
	\begin{equation*}
			\left(\| v \|_q^{\delta,\e}\right)^2 \leq \left(1 + 4 \eta (\delta) + 2 \eta^2(\delta )\right) \|v\|^2  \leq (1+2\eta(\delta))^2 \|v\|^2,
	\end{equation*}	for any $q \in \mathcal{M}_{\delta,\e}$. A similar argument ensures also that $$\|v\|^2 \leq \left(1 + 2 \eta (\delta) \right)^2 \left(\| v \|_q^{\delta,\e}\right)^2,$$ so that we finally obtain
	\begin{equation*} \label{eq:thmGC-NormEqproof}
		\frac{1}{\left(1 + 2 \eta (\delta) \right)^2} \|v\|^2 \leq \left(\| v \|_q^{\delta,\e}\right)^2 \leq \left(1 + 2 \eta (\delta) \right)^2 \|v\|^2,
	\end{equation*}
	for all $ q \in \mathcal{M}_{\delta,\e}$.
\end{proof}

	For convenience, we henceforth drop the subscript in $\|\cdot \|_q^{\delta,\e}$, instead denoting it by ${\|\cdot \|_{\delta,\e}}$, with the point of application being given implicitly. A similar abuse of notation may be adopted for the projections $\Pi_{\delta,\e}$ and $\Pi_0$, with the omission of the base-point.
	
	Let $A^{\delta,\e}_t$ and $B^{\delta,\e}_t$ denote the operators defined in \eqref{eq:defABFenichel} acting on $\mathcal{M}_{\delta,\e}$ with respect to the flow $\phi^{\delta,\e}_t$ associated to the vector field \begin{equation*}\label{eq:thmGC-Fdefin}
		F(s,x,\delta,\e) = \left(1,f(x) + \frac{\e^{\ell+1}}{\delta^\ell} \, h\left(\frac{s}{\delta^\ell},x,\e\right) \right)
	\end{equation*} of \cref{eq:continuationfamilyResAut}. Hence, if $q \in \mathcal{M}_{\delta,\e}$, then
	\begin{equation*} \label{eq:thmGC-AandB}
		\begin{aligned}
			& A^{\delta,\e}_t(q) : = D\left(\phi^{\delta,\e}_{-t}|_{\mathcal{M}_{\delta,\e}}\right) (q) \quad \text{and} \quad B^{\delta,\e}_t(q) : = \Pi_{\delta,\e} \cdot D\phi^{\delta,\e}_{t} (\phi^{\delta,\e}_{-t}(q))|_{I_{\delta,\e}}.
		\end{aligned}
	\end{equation*} 
	We also define the generalized Lyapunov type numbers $\nu^{\delta,\e}(q)$ and $\sigma^{\delta,\e}(q)$ of $\mathcal{M}_{\delta,\e}$ as in \Cref{sec:NH}. Accordingly, for $p \in \mathcal{M}_0$, we define $A_t^0(p)$ and $B_t^0(p)$ to be the same operators with respect to the flow $\phi_t^0$ associated to the vector field $F_0(x) = (1,f(x))$ of (\ref{eq:continuationfamilyResAutUnp}). The corresponding generalized Lyapunov type numbers are denoted by $\nu^0(p)$ and $\sigma^0(p)$. 
	
	Considering that we cannot simply take $\e = \delta=0$ in \eqref{eq:continuationfamilyResAut} to obtain \eqref{eq:continuationfamilyResAutUnp}, it is not immediately obvious that their associated flows must remain close. Thus, we show a very useful proximity result for $\phi_t^{\delta,\e}$ and $\phi_t^0$.
	
	\begin{lemma} \label{lemma.flowcloseness}
		Suppose that $p_0$, $p_1 \in \R^{n+1}$ and $S>0$ are such that $\phi_t^{\delta,\e}(p_1), \, \phi_t^0(p_0)$ remain in a compact set $\mathcal{K}$ for all $t \in [-S,S]$. Then, there is $C_0>0$ such that
		\begin{equation*}
			\| \phi^{\delta,\e}_t (p_1) - \phi_t^0(p_0)\| \leq C_0 \left(\|p_1 - p_0\| + \e\right),
		\end{equation*}		for $t \in [-S,S]$ and $0 \leq \e \leq \delta$. 
	\end{lemma}
	\begin{proof}
		Denoting $\phi_t^{\delta,\e} (p_1) = (s_1+t,x_t^{\delta,\e}(p_1))$ and $\phi_t^0(p_0)(s_0+t,x_t^0(p_0))$, and considering the differential equations associated to those flows, it follows that
		\begin{equation*}
			\begin{aligned}
				\|x_t^{\delta,\e}(p_1) - x_t^0(p_0)\| &\leq \|p_1 - p_0\|+ \int_0^t \|f(x_s^{\delta,\e}(p_1)) - f(x_s^0(p_0))\| ds \\ &+ \frac{\e^{\ell+1}}{\delta^\ell} \int_0^t \left\|h\left(\frac{s_1+s}{\delta^\ell},x_s^{\delta,\e}(p_1),\e\right)\right\| ds.
			\end{aligned}
		\end{equation*}		Since $f$ is of class $C^1$, the hypothesis that the flows are contained in $\mathcal{K}$ ensures there is $L>0$ such that 
		\begin{equation*}
			\| f(x_s^{\delta,\e}(p_1)) - f(x_s^0(p_0))\| \leq L \|x_s^{\delta,\e}(p_1) - x_s^0(p_0)\|,
		\end{equation*}		for $s \in [-S,S]$. Moreover, periodicity and compactness ensures that there is $M>0$ such that 
		\begin{equation*}
			\left\|h\left(\frac{s_1+s}{\delta^\ell},x_s^{\delta,\e}(p_1),\e\right)\right\| \leq M,
		\end{equation*}		for $s \in [-S,S]$. Therefore, if $\e\leq \delta$, it follows that
		\begin{equation*}
			\|x_t^{\delta,\e}(p_1) - x_t^0(p_0)\| \leq \|p_1 - p_0\| + \int_0^t L \|x_s^{\delta,\e}(p_1) - x_s^0(p_0)\| ds + \e M t.
		\end{equation*}		Hence, an application of Grönwall's inequality yields
		\begin{equation*}
			\|x_t^{\delta,\e}(p_1) - x_t^0(p_0)\| \leq (\|p_1 - p_0\| + \e M S) e^{LS}.
		\end{equation*}		Since $\|s_1 + t - s_0 +t\| = \|s_1 - s_0\| \leq \|p_1 - p_0\|$, it is easy to see that one can choose $C_0$ as indicated in the statement of the Lemma.
	\end{proof}
	
	 Take $p \in \mathcal{M}_0$ and let $q \in \mathcal{M}_{\delta,\e}$ be the unique point for which $\pi(q) = p$. In order to prove that $\mathcal{M}_{\delta,\e}$, we will estimate its generalized Lyapunov type numbers $\nu^{\delta,\e}$ and $\sigma^{\delta,\e}$ by arguing they must be close to their unperturbed counterparts $\nu^0$ and $\sigma^0$. To do so, we first estimate $\|B_t^{\delta,\e}(q)\|_{\delta,\e}$, by comparing it with $\|B_t^0(p)\|$.
	\begin{lemma} \label{lemma.Bestimate}
		For each $S>0$, there is a non-decreasing function $C_B(\delta)\geq 1$ such that $C_B(\delta) \to 1$ as $\delta \to 0^+$ and 
		\begin{equation*}
			\|B_t^{\delta,\e}(q)\|_{\delta,\e} \leq C_B(\delta) \|B_t^0(p)\|,
		\end{equation*}		for all $t \in [-S,S]$ and all $p \in \mathcal{M}_0$, where $q \in \mathcal{M}_{\delta,\e}$ is uniquely defined by $\pi(q)=p$.
	\end{lemma}
	\begin{proof}
		
		Let $S>0$ be fixed. Let $p \in \mathcal{M}_0$ and take $q \in \mathcal{M}_{\delta,\e}$ such that $\pi(q) =p$. For any $X \in I^{\delta,\e}_{q} \subset \R^{n+1}$, consider
		\begin{equation*}
			B^{\delta,\e}_t(q) \cdot X = \Pi_{\delta,\e} \cdot D\phi^{\delta,\e}_{t} (\phi^{\delta,\e}_{-t}(q)) \cdot X,
		\end{equation*} 
		which shall be compared with $B^{0}_t(p) \cdot X$. We remind the reader that, on account of the coordinates used, $X \in I^{\delta,\e}_q$ implies that $X = (0,\tilde{X}) \in \R\times \R^{n}$,
		
		It follows from \eqref{eq:thmGC-NormEq} and the triangle inequality that, for any $t \in \R$,
		\begin{equation*} \label{eq:thmGC-norms}
			\|B_t^{\delta,\e} (q) \cdot X\|_{\delta,\e} \leq \left(1+ 2 \eta (\delta)\right) \left[\|B_t^{\delta,\e} (q) \cdot X - B^{0}_t(p) \cdot X \| + \| B^{0}_t(p) \cdot X\|\right].
		\end{equation*}
		Considering the definitions of the operators $B_t^{\delta,\e}$ and $B_t^0$, and applying once more the triangle inequality, we also have 
		\begin{equation} \label{eq:thmGC-triangleS1andS2}
			\begin{aligned}
				\|B^{\delta,\e}_t(q) \cdot X - B^{0}_t(p) \cdot X\| \leq&\| \Pi_{\delta,\e} ( D\phi^{\delta,\e}_{t} (\phi^{\delta,\e}_{-t}(q)) -  D\phi^{0}_{t} (\phi^{0}_{-t}(p))) \cdot X \| \\ 
				&+ \| (\Pi_{\delta,\e} - \Pi_0)  D\phi^{0}_{t} (\phi^0_{-t}(p)) \cdot X \|,
			\end{aligned}
		\end{equation}		where $\Pi_0$ is applied on $p = \pi(q)$.
		
		Define $I(t,p,\delta,\e)$ and $II(t,p,\delta,\e)$ as the first and second summands on the right-hand side of this \eqref{eq:thmGC-triangleS1andS2}. As remarked before, by taking $\delta>0$ sufficiently small, we can make $\Pi_{\delta,\e}$ and $\Pi_0$ become arbitrarily close. Thus, by potentially redefining $\eta(\delta)$ and considering that $\mathcal{M}_0$ is compact, we can guarantee that $\|II(t,p,\delta,\e)\| \leq \eta(\delta)\|X\|$ if $\e\leq \delta$ and $t \in [-S,S]$, where $\eta(\delta) \to 0$ as $\delta \to 0^+$.
		
		For $I(t,p,\delta,\e)$, we need to show that the difference between derivatives of the flow also goes to zero with $\delta$. First, we observe that, since $\|\Pi_0\| = 1$ and $\Pi_{\delta,\e}$ is close to $\Pi_0$, there is $C_\Pi>1$ such that
		\begin{equation*}
			\|\Pi_{\delta,\e}\| \leq C_\Pi
		\end{equation*}		holds uniformly with respect to the base-point if $\delta \in (0,\delta_0)$ and $\e\leq \delta$. Hence, an application of the triangle inequality yields
		\begin{equation*} \label{eq:thmGC-ineqS1}
			\begin{aligned}
				I(t,p,\delta,\e) \leq& C_\Pi \|  D\phi^{\delta,\e}_{t} (\phi^{\delta,\e}_{-t}(q)) \cdot X -  D\phi^{0}_{t} (\pi(\phi^{\delta,\e}_{-t}(q))) \cdot X \| \\
				&+ C_\Pi \|  D\phi^{0}_{t} (\pi(\phi^{\delta,\e}_{-t}(q))) \cdot X - D\phi^{0}_{t} (\phi^{0}_{-t}(p)) \cdot X \|.
			\end{aligned}
		\end{equation*}

		We proceed to estimating the difference between the linearized flows in the perturbed and unperturbed case. Let $\p_1 h$ and $\p_2h$ denote, respectively, the partial derivatives of $h$ with respect to its first and second entry. Also, let $s_q+t$ and $x^{\delta,\e}(q)$ denote the first and second entry of the flow $\phi_t^{\delta,\e}(q)$.
		
		Observe, on the one hand, that $D\phi^{\delta,\e}_{t} (q) \cdot X$ is the solution of the first variational equation
		\begin{equation*}
			\dot z = \left[\begin{array}{cc}
				0 & 0 \\
				\frac{\e^{\ell+1}}{\delta^{2\ell}} \; \p_1 h \left(\frac{s_{q}+t}{\delta^\ell},x_t^{\delta,\e}(q),\e\right) & Df(x_t^{\delta,\e}(q)) + \frac{\e^{\ell+1}}{\delta^\ell}\, \p_2 h\left(\frac{s_{q}+t}{\delta^\ell},x_t^{\delta,\e}(q),\e\right)
			\end{array}\right] z,
		\end{equation*}		with initial condition $z(0) = X$. Therefore, since $X = (0,\tilde{X}) \in \R\times \R^n$, the first entry of $z(t)$ must also vanish for all $t \in \R$. In that case, $z(t) = (0,\tilde{z}(t)) \in \R \times \R^n$, where
		\begin{equation*}
			\dot {\tilde{z}} = \left(Df(x_t^{\delta,\e}(q)) + \frac{\e^{\ell+1}}{\delta^\ell} \partial_2 h\left(\frac{s_{q}+t}{\delta^\ell},x_t^{\delta,\e}(q),\e\right)\right) \tilde {z}, \quad \tilde {z}(0) = \tilde{X}.
		\end{equation*}		
		On the other hand, again since $X=(0,\tilde{X})$, an analogous argument ensures that $D\phi_t^0(p) \cdot X$ is of the form $w(t) = (0,\tilde{w}(t))$, with 
		\begin{equation*}
			\dot {\tilde{w}} = Df (x_t^0(p)) \tilde w, \quad \tilde{w}(0) = \tilde{X}.
		\end{equation*}		In particular, continuity and compactness yield $C_w>0$ such that 
		\begin{equation*}
			\| \tilde{w}(t)\| = \|w(t)\| = \| D\phi_t^0(p) \cdot X\| \leq C_w \|X\|,
		\end{equation*}		for all $t \in [-S,S]$ and all $p \in \mathcal{M}_0$.
		
		Defining $\tilde{u} = \tilde{z} - \tilde{w}$, it is easy to see that it must satisfy the equation
		\begin{equation*} \label{eq:thmGC-uDifEq}
			\dot{\tilde{u}} = W(t,s_{q},x_t^{\delta,\e}(q),\delta,\e) \cdot \tilde{u} + E(t,s_{q},x_t^{\delta,\e}(q),x_t^0(p),\delta,\e) \cdot \tilde{w},
		\end{equation*}
		where 
		\begin{equation*}
			W(t,s,x,\delta,\e) : = Df(x) + \frac{\e^{\ell +1}}{\delta^\ell} \p_2 h \left(\frac{s+t}{\delta^\ell},x,\e\right) 
		\end{equation*}		and 
		\begin{equation*}
			E(t,s,x,x_0,\delta,\e) : = Df(x) - Df(x_0) +  \frac{\e^{\ell +1}}{\delta^\ell} \p_2 h \left(\frac{s+t}{\delta^\ell},x,\e\right).
		\end{equation*}		Considering that $q \in \mathcal{M}_{\delta,\e}$, $p \in \mathcal{M}_0$, and that $\p_2 h$ is periodic in its first entry, it follows from the fact that the family of manifolds $\{\mathcal{M}_{\delta,\e}\}_{0\leq\e\leq \delta\leq \delta_0}$ can be contained inside a compact set that there are $L,\, M>0$, independent of the choice of $p \in \mathcal{M}_0$, such that, if $\e \leq \delta$, then 
		\begin{equation*}
			\begin{aligned}
				&\|  W(t,s_{q},x_t^{\delta,\e}(q),\delta,\e)\| \leq M, \\ &\|E(t,s_{q},x_t^{\delta,\e}(q),x_t^0(p),\delta,\e)\| \leq L \left(\|x_t^{\delta,\e}(q) - x_t^0(p)\| + \e\right).
			\end{aligned}
		\end{equation*}
		The second of those inequalities implies, using \Cref{lemma.flowcloseness}, that there is $C_1>0$ such that
		\begin{equation*}
			\|E(t,s_{q},x_t^{\delta,\e}(q),x_t^0(p),\delta,\e)\| \leq L \left[ C_0 \|q- p\| + (1+C_0) \e \right] C_w \|X\|,
		\end{equation*}
		for $t \in [-S,S]$ and $\e \leq \delta$. 
		
		Hence, by integrating \cref{eq:thmGC-uDifEq} and applying Grönwall's inequality, it follows that there is $C_1>0$ such that
		\begin{equation*}
			\|\tilde{u}(t)\| \leq C_1 \left(\|q - p\| + \e \right) \|X\|,
		\end{equation*}
		for $t \in [-S,S]$ and $\e \leq \delta$. Since $q \to p$ as $\e \to 0^+$ uniformly on $\delta$ on account of \ref{hyp.P-thmGC}, by potentially redefining $\eta(\delta)$, we guarantee that
		\begin{equation*} \label{eq:thmGC-derivativeofflowproximityX}
			\|D\phi_t^{\delta,\e}(q) \cdot X - D\phi_t^0(p)\cdot X\| \leq \eta(\delta) \|X\|
		\end{equation*}
		if $\e \leq \delta$ and $t \in [-S,S]$, where $\eta(\delta) \to 0$ as $\delta \to 0^+$. Thus, by choosing $q \to \phi_{-t}^{\delta,\e} (q)$, we can bound the first summand on the right-hand side of \cref{eq:thmGC-ineqS1} for $t \in [-S,S]$:
		\begin{equation*}
			\|  D\phi^{\delta,\e}_{t} (\phi^{\delta,\e}_{-t}(q)) \cdot X -  D\phi^{0}_{t} (\pi(\phi^{\delta,\e}_{-t}(q))) \cdot X \| \leq \eta(\delta) \|X\|.
		\end{equation*}
		
		To estimate its second summand, we observe that $\phi_t^0$ is the flow of the $C^2$ vector field $F_0(x)$, so that $D\phi^0_t(p)$ is of class $C^1$ with respect to $(t,p)$. Thus, since $\pi(\phi^{\delta,\e}_{t}(q))$ and $\phi_t^0(p)$ are in the compact set $\mathcal{M}_0$ for any $t \in \R$, it follows that there is $L'>0$, independent of $p$, such that
		\begin{equation*}
			\|  D\phi^{0}_{t} (\pi(\phi^{\delta,\e}_{-t}(q))) \cdot X - D\phi^{0}_{t} (\phi^{0}_{-t}(p)) \cdot X \| \leq L' \| \pi(\phi^{\delta,\e}_{-t}(q)) - \phi^{0}_{-t}(p)\| \|X\|,
		\end{equation*}
		for all $t \in [-S,S]$ and $p \in \mathcal{M}_0$. Then, an application of the triangle inequality yields
		\begin{equation*}
			\| \pi(\phi^{\delta,\e}_{-t}(q)) - \phi^{0}_{-t}(p)\| \leq \| \pi(\phi^{\delta,\e}_{-t}(q)) - \phi^{\delta,\e}_{-t}(q)\| + \| \phi^{\delta,\e}_{-t}(q) - \phi^{0}_{-t}(p)\|,
		\end{equation*}
		both summands tending to zero uniformly on $p$ and $\delta$ as $\e \to 0^+$, the first on account of \ref{hyp.P-thmGC} alone and the second on account of the same hypothesis combined with \Cref{lemma.flowcloseness}. Thus, by potentially redefining $\eta(\delta)$ one last time, we can guarantee that 
		\begin{equation*}
			\|I(t,p,\delta,\e) \| \leq \eta(\delta) \|X\|
		\end{equation*}
		if $\e \leq \delta$ and $t \in [-S,S]$, where $\eta(\delta) \to 0$ as $\delta \to 0^+$.
		
		Taking into account \cref{eq:thmGC-norms}, \cref{eq:thmGC-triangleS1andS2} and the estimates found for $I$ and $II$, we obtain
		\begin{equation*} \label{eq:thmGC-Best1}
			\begin{aligned}
				\|B_t^{\delta,\e} (q) \cdot X\|_{\delta,\e} &\leq \left(1+ 2 \eta (\delta)\right) \left[3 \eta(\delta) \|X\| + \| B^{0}_t(p) \cdot X\|\right],
			\end{aligned}
		\end{equation*}
		for $t \in [-S,S]$ and $p \in \mathcal{M}_0$.
		
		Define
		\begin{equation*}
			m^0_{B,S} := \inf \{\|B_t^0(p)\|: t \in [-S,S], p \in \mathcal{M}_0 \} > 0 
		\end{equation*}
		and observe that it follows from \cref{eq:thmGC-Best1} that
		\begin{equation*}\label{eq:thmGC-BboundC}
			\frac{\|B_t^{\delta,\e} (q) \cdot X\|_{\delta,\e}}{\|X\|_{\delta,\e}} \leq (1+2\eta(\delta))^2 \left[3\eta(\delta) + \|B_t^0(p)\| \right] \leq C_B(\delta) \|B_t^0(p)\|,
		\end{equation*}
		for all $t \in [-S,S]$ and all $p \in \mathcal{M}_0$, where 
		\begin{equation*}
			C_B(\delta) : = (1+2\eta(\delta))^2 \left(1 + \frac{3 \eta(\delta)}{m_{B,S}^0}\right)
		\end{equation*}
		is such that $C(\delta)\to 1$ as $\delta \to 0^+$.
	\end{proof}

	We continue the proof by estimating the generalized Lyapunov type number $\nu^{\delta,\e}$. Considering that our proximity results only hold provided that we choose a compact interval $[-S,S]$ for $t$ to be in, we need to estimate the limits with $t \to \infty$ appearing in the definitions of normal hyperbolicity with finite-time expressions. This is achieved via submultiplicativity, as follows.
	
	Since the tangent bundle $T\mathcal{M}_{\delta,\e}$ is invariant under $D\phi^{\delta,\e}_t$, it follows by an application of the chain rule that
	\begin{equation*} \label{eq:thmGC-Bcomp}
		B^{\delta,\e}_{t+s}(q) = \Pi_{\delta,\e} \cdot D\phi^{\delta,\e}_{t} (\phi^{\delta,\e}_{-t}(q)) \cdot \Pi_{\delta,\e} \cdot  D\phi_s^{\delta,\e}(\phi^{\delta,\e}_{-t-s}(q))|_{I_{\delta,\e}},
	\end{equation*}
	for any $t,s \in \R$. Thus, by letting $\|B_t^{\delta,\e}\|_{\delta,\e} :=\sup\{\|B_t^{\delta,\e}(q)\|_{\delta,\e}:q \in \mathcal{M}_{\delta,\e}\}$, we obtain
	\begin{equation*} \label{eq:thmGC-Bchainrule}
		\|B^{\delta,\e}_{t+s}\|_{\delta,\e} \leq \|B_t^{\delta,\e}\|_{\delta,\e} \times \|B_s^{\delta,\e}\|_{\delta,\e},
	\end{equation*}
	for any $t,s \in \R$.
	
	Fix $S>0$ and write any $t>0$ as $ t = Q S + R$, where $ Q \in \mathbb{N}$ and $R \in [0,S)$. Then, applying \cref{eq:thmGC-Bchainrule} multiple times, we have
	\begin{equation*}
		\|B_t^{\delta,\e}\|_{\delta,\e}^{\frac{1}{t}} \leq \|B_S^{\delta,\e}\|_{\delta,\e}^{\frac{Q}{QS + R}} \, \|B_{R}^{\delta,\e}\|_{\delta,\e}^{\frac{1}{QS + R}}.
	\end{equation*}
	Taking the $\limsup$ as $t \to +\infty$, i.e., as $Q \to +\infty$, it follows from \Cref{lemma.lypaunovnumberseqdef} that
	\begin{equation} \label{eq:thmGC-ineqmu}
		\nu^{\delta,\e}(q) = \limsup_{t\to +\infty} \|B_t^{\delta,\e}(q)\|_{\delta,\e}^{\frac{1}{t}} \leq \|B_S^{\delta,\e}\|_{\delta,\e}^{\frac{1}{S}},
	\end{equation}
	for any $q \in \mathcal{M}_{\delta,\e}$. We remark that this inequality holds for any choice of $S>0$. Instead of calculating $\nu^{\delta,\e}(q)$ directly, we will find $S>0$ such that the estimate above already guarantees relative normal hyperbolicity for small $\delta$.

	Since $M_0$ is relatively $r_H$-normally hyperbolic by hypothesis, so must be $\mathcal{M}_0$. Thus, $\nu^0(p) <1$ for every $p \in \mathcal{M}_0$, and the \hyperref[lemma.uniformitylemma]{Uniformity Lemma} guarantees that there are $\kappa \in (0,1)$ and $c_0>0$ such that $\|B_t^0\| \leq c_0 \kappa^t$ for all $t \geq 0$. In particular, by taking $S>1$ sufficiently large, we can ensure that 
	\begin{equation*}
		\|B_S^0\|^{\frac{1}{S}} \leq c_0^{\frac{1}{S}} \kappa < 1,
	\end{equation*}
	where $\|B_t^{0}\| :=\sup\{\|B_t^0(p)\|:p \in \mathcal{M}_{0}\}$.	
	
	Hence, applying \Cref{lemma.Bestimate} to \eqref{eq:thmGC-ineqmu} with this choice of the constant $S$, it follows at once that, for any $q \in \mathcal{M}_{\delta,\e}$, if $\e \leq \delta$ and $\delta$ is sufficiently small, then
	\begin{equation*}
		\nu^{\delta,\e}(q) \leq C(\delta)^{\frac{1}{S}} \|B_S^0\|^{\frac{1}{S}} \leq \left(C(\delta) c_0\right)^{\frac{1}{S}} \kappa.
	\end{equation*}
	It is easy to see that the right-hand side of this equation can be made arbitrarily close to $c^\frac{1}{S}_{0} \kappa$ by taking smaller $\delta$. Consequently, if $\e \leq \delta$ and $\delta$ is sufficiently small, then
	\begin{equation*}
		\nu^{\delta,\e} (q) < 1,
	\end{equation*}
	for all $q \in \mathcal{M}_{\delta,\e}$. We remark that the upper bound for $\nu^{\delta,\e}$ can be made as close to $\kappa$ as wanted by first choosing $S>1$ and then adequately restricting the size of $\delta$.

	Henceforth, we consider the second Lyapunov type number $\sigma^{\delta,\e}(q)$. For this, we need concern ourselves with vectors that are tangent to $\mathcal{M}_{\delta,\e}$. Let $Q^{\delta,\e} \subset T\mathcal{M}_{\delta,\e}$ be the bundle of vectors having zero as their first entry. Considering that the foliation whose leaves are given by $s = \textit{constant}$ is invariant under the flow of \cref{eq:continuationfamilyResAut}, it follows that $Q^{\delta,\e}$ is invariant under the linearised flow. Furthermore, by letting $J^{\delta,\e}$ denote the vector bundle generated by the vector field $F(s,x,\delta,\e)$ appearing in \cref{eq:continuationfamilyResAut}, it is easy to see that 
	\begin{equation*}
		T \mathcal{M}_{\delta,\e} = Q^{\delta,\e} \oplus J^{\delta,\e}. 
	\end{equation*}

	As before, let $\Pi^Q_{\delta,\e}:T\mathcal{M}_{\delta,\e} \to Q^{\delta,\e}$ and $\Pi^J_{\delta,\e}: T\mathcal{M}_{\delta,\e} \to J^{\delta,\e}$ denote the projections corresponding to the above-mentioned splitting. It is clear that, since $Q_q^{\delta,\e}$ and $J^{\delta,\e}$ tend to $Q^0_p$ and $J^0_p$ uniformly because of \ref{hyp.P-thmGC} and compactness of $\mathcal{M}_0$, there is $C_{\|\cdot\|}>0$ such that, for any $\delta \in (0,\delta_0)$, $\e \leq \delta$, $q \in \mathcal{M}_{\delta,\e}$, $Y \in Q_q^{\delta,\e}$, and $Z \in J_q^{\delta,\e}$, the following holds:
	\begin{equation*} \label{eq:thmGC-normestimate}
		\|Y\| +\|Z\| \leq \|\Pi^Q_{\delta,\e}(q) \cdot (Y+Z)\| +  \|\Pi^J_{\delta,\e}(q) \cdot (Y+Z)\| \leq C \| Y+Z\|.
	\end{equation*}
	
	In order to proceed, we prove an estimate similar to the one presented in \Cref{lemma.Bestimate} but concerning the operator $A^{\delta,\e}_t$.
	
	\begin{lemma} \label{lemma.Aestimate}
		For each $S>0$, there is a non-decreasing function $C_A(\delta)\geq 1$ such that $C_A(\delta) \to 1$ as $\delta \to 0^+$ and 
		\begin{equation*}
			\|A_t^{\delta,\e}(q)\|_{\delta,\e} \leq C_A(\delta) \|A_t^0(p)\|,
		\end{equation*}
		for all $t \in [-S,S]$ and all $p \in \mathcal{M}_0$, where $q \in \mathcal{M}_{\delta,\e}$ is uniquely defined by $\pi(q)=p$.
	\end{lemma}
	\begin{proof}
	Let $S>0$ be fixed. Let $p \in \mathcal{M}_0$ and take $q \in \mathcal{M}_{\delta,\e}$ such that $\pi(q) =p$. Since every vector in $Q^{\delta,\e}$ has vanishing first entry, an argument similar to the one employed in \Cref{lemma.Bestimate} can be used to obtain
	\begin{equation*} \label{eq:thmGC-derivativeofflowproximityY}
		\|D\phi_t^{\delta,\e}(q) \cdot Y - D\phi_t^0(p)\cdot Y\| \leq \eta(\delta) \|Y\|
	\end{equation*}
	if $\e \leq \delta$, $t \in [-S,S]$ and $Y \in Q_q^{\delta,\e}$, where $\eta(\delta) \to 0$ as $\delta \to 0^+$. It is easy to see that, by definition of $A_t^{\delta,\e}$, this is equivalent to 
	\begin{equation*} \label{eq:thmGC-AestiY}
		\|A_t^{\delta,\e}(q) \cdot Y - A_t^0(p)\cdot Y\| \leq \eta(\delta) \|Y\|,
	\end{equation*}
	for $t \in [-S,S]$.

	All that remains is investigating how $A_t^{\delta,\e}$ acts on $J^{\delta,\e}$. Observe that any vector $Z \in J_q^{\delta,\e}$ must be of the form $\lambda F(q,\delta,\e)$, with $\lambda \in \R$. Likewise, any vector $Z_0 \in J_p^0$ is of the form $\lambda F_0(p)$, with $\lambda \in \R$. Thus we will focus on studying the action of $A_t^{\delta,\e}$ on $F$, comparing it with the action of $A_t^0$ on $F_0$.
	
	Since the vector field is always a solution of the corresponding first variational equation, it follows that
	\begin{equation*}
		A_{t}^{\delta,\e}(q) \cdot F(q,\delta,\e) = F(\phi^{\delta,\e}_{-t}(q),\delta,\e), \quad  \text{and} \quad A_{t}^{\delta,\e}(p) \cdot F_0(p) = F_0(\phi^{0}_{-t}(p)).
	\end{equation*}
	for any $t \in \R$. Therefore,
	\begin{equation*}
		\begin{aligned}
			\left\| A_{t}^{\delta,\e}(q) \cdot F(q,\delta,\e) - A_{t}^{\delta,\e}(p) \cdot F_0(p)  \right\| \leq& \|F(\phi^{\delta,\e}_{-t}(q),\delta,\e) - F_0(\phi^{\delta,\e}_{-t}(q))\| \\ 
			&+\| F_0(\phi^{\delta,\e}_{-t}(q)) -  F_0(\phi^{0}_{-t}(p))\|
		\end{aligned}
	\end{equation*}
	which, on account of compactness of $\{\mathcal{M}_{\delta,\e}\}_{0\leq\e\leq \delta\leq \delta_0}$, periodicity of $h$ in its first entry, and \Cref{lemma.flowcloseness}, implies that there is $L''>0$, independent of $p \in \mathcal{M}_0$ such that
	\begin{equation*}
		\begin{aligned}
				\left\| A_{t}^{\delta,\e}(q) \cdot F(q,\delta,\e) - A_{t}^{\delta,\e}(p) \cdot F_0(p)  \right\| \leq L'' \left(\frac{\e^{\ell+1}}{\delta^\ell} + \e + \|q-p\| \right),
		\end{aligned}
	\end{equation*}
	for any $t \in [-S,S]$ and any $p \in \mathcal{M}_0$. Finally, considering \ref{hyp.P-thmGC} and by potentially redefining $\eta(\delta)$ once again, we obtain
	\begin{equation*}\label{eq:thmGC-AestiJ}
			\left\| A_{t}^{\delta,\e}(q) \cdot F(q,\delta,\e) - A_{t}^{\delta,\e}(p) \cdot F_0(p)  \right\| \leq \eta(\delta) \leq \eta(\delta) \|F(q,\delta,\e)\|,
	\end{equation*}
	if $\e \leq \delta$ and $t \in [-T,T]$, because $\|F(q,\delta,\e)\| \geq 1$ for all $q \in \mathcal{M}_{\delta,\e}$.
	
	We can thus estimate $\|A_t^{\delta,\e}\|_{\delta,\e}$ for $t \in [-S,S]$, by taking $Y+\lambda F(q,\delta,\e) \in Q_q^{\delta,\e} \oplus J^{\delta,\e}$, with $\lambda \in \R$, and writing
	\begin{equation*}
		\begin{aligned}
			\|A_t^{\delta,\e}(q) \cdot (Y+ \lambda F(q,\delta,\e)) \|_{\delta,\e} \leq& (1+2\eta(\delta)) \|A_t^{\delta,\e}(q) \cdot Y - A_t^0(p) \cdot Y\| \\
			&+(1+2\eta(\delta))\lambda \|A_t^{\delta,\e}(q)\cdot F(q,\delta,\e) - A_t^0(p) F_0(p) \| \\
			&+\|A_t^0(p) \cdot (Y+ \lambda F_0(p))\|.
		\end{aligned}
	\end{equation*}
	Taking into account the estimates \cref{eq:thmGC-AestiY,eq:thmGC-AestiJ}, it follows that 
	\begin{equation*}
		\begin{aligned}
			\|A_t^{\delta,\e}(q) \cdot (Y+ \lambda F(q,\delta,\e)) \|_{\delta,\e} \leq& (1+2\eta(\delta)) \eta(\delta) \left(\|Y\| + \|\lambda F(q,\delta,\e)\|\right)   \\ &+ \|A_t^0(p) \cdot (Y+ \lambda F_0(p))\|.
		\end{aligned}
	\end{equation*}
	Then, by dividing by $\|Y+\lambda F(q,\delta,\e)\|_{\delta,\e}$ and considering \cref{eq:thmGC-NormEq,eq:thmGC-normestimate}, it follows that
	\begin{equation*} \label{eq:thmGC-estimateA.1}
		\|A_t^{\delta,\e}(q) \|_{\delta,\e} \leq \eta(\delta) (1+2\eta(\delta))^2 + \|A_t^0(p)\| \frac{\|Y+\lambda F_0(p)\|}{\|Y+\lambda F(q,\delta,\e)\|}.
	\end{equation*}
	The fraction on the right-hand side satisfies
	\begin{equation*}
		\frac{\|Y+\lambda F_0(p)\|}{\|Y+\lambda F(q,\delta,\e)\|} \leq 1+ \frac{\|\lambda F_0(p) - \lambda F(q,\delta,\e)\|}{\|Y + \lambda F(q,\delta,\e)\|} \leq 1 + \lambda \frac{\| F_0(p) - F(q,\delta,\e)\| }{\lambda},
	\end{equation*}
	which certainly tends to 1 uniformly on $p$ and $\delta$ as $\e \to 0^+$. Therefore, since $m^0_{A,S}:=\inf \{\|A_t^0(p)\| : t \in [-S,S], p \in \mathcal{M}_0\} >0$, we can proceed as in \Cref{lemma.Bestimate} and conclude from \cref{eq:thmGC-estimateA.1} that, if $\e \leq \delta$, then
	\begin{equation*} \label{eq:thmGC-AboundC}
		\|A_t^{\delta,\e}(q) \|_{\delta,\e} \leq C_A(\delta) \|A_t^0(p)\|,
	\end{equation*}
	for all $t \in [-S,S]$ and all $p \in \mathcal{M}_0$, where $C_A(\delta) \to 1$ as $\delta \to 0^+$.
	\end{proof}
	
	We proceed to estimating $\sigma^{\delta,\e}(q)$. Since $\sigma^0(p) < 1/r_H$ for any $p \in \mathcal{M}_0$ because of \ref{hyp.2'-thmGC}, the \hyperref[lemma.uniformitylemma]{Uniformity Lemma} ensures that there is $b < 1/r_H$ such that
	\begin{equation*}
		\lim_{t \to \infty} \|A_t^0(p)\| \|B_t^0(p)\|^b = 0
	\end{equation*} 
	uniformly for $p \in \mathcal{M}_0$. Hence, there is $S>1$ such that 
	\begin{equation*} \label{eq:thmGC-Tchoicesigma}
		\|A_S^0(p)\| \|B_S^0(p)\|^b \leq \frac{1}{4},
	\end{equation*}
	for all $p \in \mathcal{M}_0$.
	
	For $i \in \mathbb{N}$, let $q_i := \phi_{-iT}(q) \in \mathcal{M}_{\delta,\e}$ and $p_i:=\pi(q_i)$. Once again, we write any $t>0$ as $t = QS +R$, where $Q \in \N$ and $R \in [0,S)$. Reasoning as we did in \cref{eq:thmGC-Bcomp}, we obtain
	\begin{equation*}
		\|B_t^{\delta,\e}(q)\|_{\delta,\e} \leq \|B_S^{\delta,\e}(q_0)\|_{\delta,\e} \cdots \|B_S^{\delta,\e}(q_{Q-1})\|_{\delta,\e} \|B_R^{\delta,\e}(q_Q)\|_{\delta,\e}.
	\end{equation*}
	A similar argument guarantees that 
	\begin{equation*}
		\|A_t^{\delta,\e}(q)\|_{\delta,\e} \leq \|A_S^{\delta,\e}(q_0)\|_{\delta,\e} \cdots \|A_S^{\delta,\e}(q_{Q-1})\|_{\delta,\e} \|A_R^{\delta,\e}(q_Q)\|_{\delta,\e}.
	\end{equation*}
	Therefore, taking into account \Cref{lemma.Bestimate,lemma.Aestimate}, it follows that 
	\begin{equation*} \label{eq:thmGC-sigmaestimate1}
		\|A_t^{\delta,\e}(q)\|_{\delta,\e} \|B_t^{\delta,\e}(q)\|^b_{\delta,\e}  \leq \left(C(\delta)\right)^{(b+1)(Q+1)} \|A_R^0(p_Q)\| \|B_R^0(p_Q)\|^b \, P^b_{p,Q},
	\end{equation*}
	where $P^b_{p,Q}$ is defined by
	\begin{equation*}
		P^b_{p,Q}: = \prod_{i=0}^{Q-1} \|A_S^0(p_i)\| \|B_S^0(p_i)\|^b.
	\end{equation*}
	
	Considering \cref{eq:thmGC-Tchoicesigma}, it follows at once that 
	\begin{equation*}
		P^b_{p,Q} \leq \frac{1}{4^Q}.
	\end{equation*}
	Moreover, since $p_Q$ is in the compact set $\mathcal{M}_0$, and since $R \in [0,S]$, there is $C_{b}>0$ such that
	\begin{equation*}
		\|A_R^0(p)\| \|B_R^0(p)\|^b \leq C_b,
	\end{equation*}
	for all $p \in \mathcal{M}_0$ and all $R \in [0,S]$.
	Thus, it follows from \cref{eq:thmGC-sigmaestimate1} that
	\begin{equation*}
		\|A_t^{\delta,\e}(q)\|_{\delta,\e} \|B_t^{\delta,\e}(q)\|^b_{\delta,\e}  \leq C_b C(\delta)^{(b+1)} \left(\frac{C(\delta)^{(b+1)}}{4}\right)^Q.
	\end{equation*}
	By choosing $\delta$ sufficiently small, we can guarantee $C(\delta)^{(b+1)} <2$. Hence, considering that $Q \to +\infty$ as $t \to +\infty$, we obtain
	\begin{equation*}
		\lim_{t \to +\infty} \|A_t^{\delta,\e}(q)\|_{\delta,\e} \|B_t^{\delta,\e}(q)\|^b_{\delta,\e} = 0,
	\end{equation*}
	which ensures that $\sigma^{\delta,\e}(q) \leq b <1/r_H$ for all $q \in \mathcal{M}_{\delta,\e}$. 
	
	We have thus proved that $\nu^{\delta,\e}(q) <1$ and $\sigma^{\delta,\e}(q)<1/r_H$ for all $q \in \mathcal{M}_{\delta,\e}$ if $\delta$ is sufficiently small and $\e \leq \delta$. By definition, $\mathcal{M}_{\delta,\e}$ is thus relatively $r_H$-normally hyperbolic. This concludes the proof of the relative case by the argument laid out in the first paragraph of the proof.
	
	The case of absolute normal hyperbolicity is proved similarly, by considering the inequality 
	\begin{equation*}
		\nu^{\delta,\e}(q) \leq \|B_S^{\delta,\e}\|_{\delta,\e}^{\frac{1}{S}} \leq \left(C(\delta)\right)^{\frac{1}{S}} \|B_S^0\|^{\frac{1}{S}} \leq \left(C(\delta)\right)^{\frac{1}{S}} K_B^{\frac{1}{S}} e^{-\rho_N},
	\end{equation*}
	which follows from \cref{eq:thmGC-BboundC}, \Cref{lemma.Bestimate} and the definition of absolute normal hyperbolicity. Then, for each $\tilde{\rho}_N<\rho_N$, we can choose $S>0$ sufficiently large and $\delta>0$ sufficiently small such that $\nu^{\delta,\e}(q) \leq e^{-\tilde{\rho}_N}$ for all $q \in \mathcal{M}_{\delta,\e}$. Reasoning likewise for $A_t^{\delta,\e}$, we conclude that, for any $\tilde{\rho}_M >\rho_M$, we can choose $S$ and $\delta$ such that $\xi_{\pm}^{\delta,\e}(q) \leq e^{\tilde{\rho}_M}$ as well. Since \ref{hyp.2'-thmGC} ensures that $r_H \rho_M<\rho_N$, the new exponents can be chosen to satisfy $r_H \tilde{\rho}_M < \tilde{\rho}_N$, in which case $\mathcal{M}_{\delta,\e}$ is absolutely $r_H$-normally hyperbolic if $\e\leq\delta$ by \Cref{lemma.verifyabsNH}.

\section{Invariant manifolds in singular families of differential systems} \label{sec:existingresultshenry}

The application of \Cref{thm:GC} is contingent on \textit{a priori} knowledge of the existence of a certain family of invariant manifolds. In the case of the tori treated in this paper, this knowledge is provided by a result ultimately due to Henry \cite{Henry1981}, but which can be traced back to the works of Bogolyubov and Mitropolsky \cite{BM}, and Hale \cite{hale61,Hale}. In this section, we present a version adapted to our setting.

Let $k,n \in \mathbb{N}^*$ and denote by $\mathcal{B}_n(x_0,\rho) \subset \R^n$ the open ball centered at $x_0$ with radius $\rho>0$. Consider the following family of differential systems:
\begin{equation} \label{system.statement}
	\begin{aligned}
		&\dot \theta = w(\theta) + \Theta (t,\theta, {\bf h} ,\delta,\e),\\
		&\dot {\bf h} = H(\theta)\cdot {\bf h} + \zeta (t,\theta, {\bf h},\delta,\e),
	\end{aligned}
\end{equation}
where $\rho_0,\delta_0,\e_0>0$, $E(\delta_0,\e_0) = \{(\delta,\e) \in (0,\delta_0) \times (0,\e_0):\e \leq \delta \}$, and $(\theta,{\bf h}, \delta, \e) \in \R^k \times \mathcal{B}_{n}(0,\rho_0) \times E(\delta_0,\e_0)$. For convenience, we will let $Y(t,t_0,\theta_0)$ denote the solution of the initial value problem
\begin{equation*}
	\dot \theta = w(\theta), \quad \theta(t_0) = \theta_0.
\end{equation*} 
The following hypotheses are assumed to hold for (\ref{system.statement}):

\begin{enumerate}[label=$(E_\arabic*$)]
	\item \label{hyp.thmhenry.thetaperiodic}  $w$, $H$, $\Theta$, and $\zeta$ are periodic in each entry of $\theta \in \R^k$, with period vector given by $$\omega =(\omega_1,\ldots,\omega_k) \in \R_+^k;$$
	\item \label{hyp.thmhenry.C2andbounded} For each $(\delta,\e) \in E(\delta_0,\e_0)$ fixed, the functions $w$, $H$, $\Theta$, and $\zeta$ are of class $C^{2}$ on $$\R \times \R^k \times \mathcal{B}_n(0,\rho_0);$$
	\item \label{hyp.thmhenry.uniformbound} $w$, $H$, $\Theta$, and $\zeta$, as well as their partial derivatives up to the second order with respect to $\theta$ and ${\bf h}$, are uniformly bounded by a constant $M>0$ on the domain $$(t,\theta,{\bf h},\delta,\e) \in \R \times \R^k \times \mathcal{B}_n(0,\rho_0) \times E(\delta_0,\e_0);$$
	\item \label{hyp.thmhenry.tangentbehavior} There are $\beta\geq0$ and $K_\theta\geq 1$ such that, for any $\theta_1, \theta_2 \in \R^k$, the inequality 
	\begin{equation*}
		\|\Upsilon(t,t_0,\theta_1) - \Upsilon(t,t_0,\theta_2)\| \leq K_\theta e^{\beta |t-t_0|} \|\theta_1 - \theta_2\|
	\end{equation*}
	holds for any $t, t_0 \in \R$;
	\item \label{hyp.thmhenry.normalbehavior}	There are $\alpha>0$ and $K_H \geq 1$ such that, for each $(t_0,\theta_0) \in \R \times \R^k$, the matrix solution $\Phi(t,s,t_0,\theta_0)$ of the non-autonomous linear differential equation $\dot {\bf v} = H(\Upsilon(t,t_0,\theta_0)) \cdot {\bf v}$ having initial condition $\Phi(s,s,t_0,\theta_0) = \Id \in \R^{n \times n}$ satisfies
	$$\|\Phi(t,s,t_0,\theta_0)\| \leq K_H e^{-\alpha (t-s)}, \quad \text{for} \; t\geq s. $$
	\item \label{hyp.thmhenry.perturbationsize} There is a continuous non-decreasing function $q:(0,\e_0) \to \R_+$ such that $q(\e) \to 0$ as $\e \to 0^+$ and, for each fixed $\e_* \in (0,\e_0)$, the following inequalities hold over $(t,\theta,\delta,\e) \in \R \times \R^k \times E(\delta_0,\e_*)$:
	\begin{enumerate}[label=\roman*.]
		\item $\|\Theta(t,\theta,0,\e)\|, \left \|\frac{\p \Theta}{\p \theta}(t,\theta,0,\e)\right\|  \leq q(\e_*),$ 
		\item $\|\zeta(t,\theta,0,\e)\|,\left \|\frac{\p \zeta}{\p \theta}(t,\theta,0,\e)\right\|, \left \|\frac{\p \zeta}{\p {\bf h}}(t,\theta,0,\e)\right\| \leq q(\e_*).$
	\end{enumerate}
\end{enumerate}

\begin{theorem}[Theorem 9.1.1, \cite{Henry1981}] \label{thm.henry}
	Suppose $\alpha>\beta$ and let $\mu>0$  be such that $\beta(1+\mu) <\alpha$. There is $\e_1 \in (0,\e_0)$ such that:
	\begin{enumerate} [label=\textit{(\alph*)}]
		\item For each $(\delta,\e) \in E(\delta_0,\e_1)$, there is a continuous function $\sigma_{\delta,\e}: \R \times \R^k \to \mathcal{B}_n(0,\rho_0)$, differentiable in its second entry, such that $\mathcal{C}_{\delta,\e}:= \{ (\theta, \sigma_{\delta,\e}(t,\theta)) :(t,\theta) \in \R \times \R^k\}$ is an invariant manifold of (\ref{system.statement});
		\item $\mathcal{C}_{\delta,\e}$ is asymptotically stable and unique in the following sense: there is $\rho_U>0$ such that any larger invariant subset containing $\mathcal{C}_{\delta,\e}$ also contains a solution for which $\|{\bf h}(t)\| \geq \rho_U$ for some $t \in \R$;
		\item $ \frac{\p \sigma_{\delta,\e}}{\p \theta}$ is uniformly Hölder continuous with exponent $\mu$;
		\item \label{thmitem.BDelta} There are continuous functions $B, \Delta: (0,\e_1] \to \R_+$ approaching $0$ as $\e \to 0^+$ such that $\|\sigma_{\delta,\e}\| \leq B(\e)$ and $\left\|\frac{\p \sigma_{\delta,\e}}{\p \theta}\right\| \leq \Delta(\e)$;
		\item \label{thmitem.thetaperiodic} $\sigma_{\delta,\e}$ is periodic in each entry of $\theta$, with the same period vector $\omega$ given in \ref{hyp.thmhenry.thetaperiodic};
		\item \label{thmitem.periodic} If $T_{\delta,\e}>0$ is such that each function appearing in (\ref{system.statement}) is $T_{\delta,\e}$-periodic in $t$, then so is 
		$\sigma_{\delta,\e}$. 		
	\end{enumerate}
\end{theorem}

The version of \Cref{thm.henry} appearing in the provided reference is much more general, allowing for the domain of the differential equation to be any Banach space, thus requiring an array of additional technical assumptions that are mostly trivial in our setting. Still, some comments comparing the original hypotheses with the ones adopted here are warranted. 

First, since \cite[Theorem 9.1.1]{Henry1981} does not consider a parameterized family of systems directly, we require that every assumption holds uniformly throughout the family in our setting---cf. \cite[Theorem 9.1.9]{Henry1981}, where stronger regularity with respect to parameters is assumed. Secondly, our assumption that all functions are of class $C^2$ and periodic in $\theta$ guarantees that, by fixing $t$, $\delta$, and $\e$, their first order derivatives with respect to the $\theta$ and ${\bf h}$ are uniformly Lipschitz continuous on the bounded domain corresponding to those variables. Hence, they are also uniformly Hölder continuous on this domain for any exponent $\mu \in (0,1]$. Finally, the assumption of ``integral-smallness'' given in the original reference is replaced here with the stronger upper bounds in \ref{hyp.thmhenry.perturbationsize}.

\subsection{Regularity of $\sigma_{\delta,\e}$ with respect to $t$}
The conclusions of \Cref{thm.henry} can only be employed as the tool to verify \ref{hyp.P-thmGC} in the case studied if the function $\sigma_{\delta,\e}$ provides an invariant manifold of class $C^1$ tending to ${\bf h} =0$ uniformly on $\delta$ as $\e \to 0^+$. However, only regularity with respect to $\theta$ is obtained directly from the theorem. Here, we show that this is sufficient, by proving that $\sigma_{\delta,\e}$ also has a continuous partial derivative with respect to $t$, which moreover goes to zero with $\e$.

\begin{lemma}\label{lemma.sigma.tdiff}
	Under the hypotheses of \Cref{thm.henry}, $\sigma_{\delta,\e}$ admits a continuous partial derivative with respect to $t$. Furthemore, this derivative satisfies the identity
	\begin{equation*} \label{eq:psigmaptidentity}
		\frac{\p \sigma_{\delta,\e}}{\p t} (t,\theta) = H(\theta) \sigma_{\delta,\e} (t,\theta) + \zeta(t,\theta,\sigma_{\delta,\e}(t,\theta),\delta,\e) - \frac{\p \sigma}{\p \theta }(t,\theta) \left[w(\theta) + \Theta (t,\theta, \sigma_{\delta,\e}(t,\theta),\delta,\e)\right].
	\end{equation*}
	In particular, $\|\sigma_{\delta,\e}\|_{C^1} \to 0$ uniformly on $\delta$ as $\e \to 0^+$.
\end{lemma}

\begin{proof}
	Consider the restricted flow $\phi_t(\tau_0,\theta_0)$ on the invariant manifold ${\bf h} = \sigma_{\delta,\e}(\tau,\theta)$ of the autonomous version of \cref{system.statement}, i.e., the solution of the initial value problem
	\begin{equation*}
		\dot \tau =1, \quad \dot \theta = w(\theta) + \Theta(\tau,\theta,\sigma_{\delta,\e}(\tau,\theta),\delta,\e); \quad \tau(0) = \tau_0, \quad \theta(0) = \theta_0.
	\end{equation*}
	For each fixed $t \in \R$, the fact that $\phi_t$ is a flow guarantees that it is invertible. Thus, $(t,\theta) \mapsto \phi_t(0,\theta)$ must also be invertible. In fact, $\dot \tau =1$ implies that $\phi_{t_1}(0,\theta_1) = \phi_{t_2}(0,\theta_2)$ only if $t_1 = t_2$, which then implies $\theta_1 = \theta_2$ as well. Define $i(t,\theta) = (i_1(t,\theta),i_2(t,\theta)) \in \R \times \R^k$ as the inverse of $(t,\theta) \mapsto \phi_t(0,\theta)$. By the Inverse Function Theorem, $i_1$ and $i_2$ are of class $C^1$.
	
	Let $(\theta_{\delta,\e}(t,t_0,\theta_0,{\bf h_0}),{\bf h}_{\delta,\e}(t,t_0,\theta_0,{\bf h_0}))$ denote the solution of \cref{system.statement} satisfying the initial conditions $\theta_{\delta,\e}(t_0,t_0,\theta_0,{\bf h_0}) = \theta_0$ and ${\bf h}_{\delta,\e}(t_0,t_0,\theta_0,{\bf h_0}) = {\bf h_0}$. Invariance of ${\bf h} = \sigma_{\delta,\e}(t,\theta)$ with respect to \cref{system.statement} ensures that
	\begin{equation*} \label{eq:lemma.sigma.tdiff-invariance}
		{\bf h}_{\delta,\e} (t,t_0,\theta_0,\sigma_{\delta,\e}(t_0,\theta_0)) = \sigma_{\delta,\e} (t+ t_0, \theta_{\delta,\e}(t,t_0,\theta_0,\sigma_{\delta,\e}(t_0,\theta_0))),
	\end{equation*}
	for any $t,t_0 \in \R$ and $\theta_0 \in \R^k$. In particular, fixing $(t,\theta) \in \R \times \R^k$ and taking $t_0=0$ and $\theta_0 = i_2(t,\theta)$  it follows that 
	\begin{equation*}\label{eq:lemma.sigma.tdiff-invariance2}
		{\bf h}_{\delta,\e} \Big(t,0,i_2(t,\theta),\sigma_{\delta,\e}(0,i_2(t,\theta))\Big) = \sigma_{\delta,\e}\Big(t,\theta_{\delta,\e}\big(t,0,i_2(t,\theta),\sigma_{\delta,\e}(0,i_2(t,\theta))\big)\Big).
	\end{equation*} 
	Hence, considering that 
	\begin{equation*}
		\phi_t(0,\theta) = \Big(t, \theta_{\delta,\e}(t,0,\theta,\sigma_{\delta,\e}(0,\theta))\Big),
	\end{equation*}
	on account of the fact that both sides are solutions of the same initial value problem, it follows that \cref{eq:lemma.sigma.tdiff-invariance2} can be rewritten as
	\begin{equation*}\label{eq:lemma.sigma.tdiff-invariance3}
		{\bf h}_{\delta,\e} \Big(t,0,i_2(t,\theta),\sigma_{\delta,\e}(0,i_2(t,\theta))\Big) = \sigma_{\delta,\e}(t,\theta).
	\end{equation*} 
	It is thus clear that $\frac{\p \sigma_{\delta,\e}}{\p t}$ exists.
	
	We are then allowed to differentiate \cref{eq:lemma.sigma.tdiff-invariance} with respect to $t$ at $t=t_0$, yielding \cref{eq:psigmaptidentity}. Thus, taking the norm on both sides of this equation, it follows from \ref{hyp.thmhenry.C2andbounded} and \ref{hyp.thmhenry.perturbationsize} that
	\begin{equation*}
		\left \|\frac{\p \sigma_{\delta,\e}}{\p t}\right\| \leq  2 M B(\e) + q(\e) + 2M\Delta(\e),
	\end{equation*}
	which certainly tends to zero as $\e \to 0^+$. Combined with the estimates $\|\sigma_{\delta,\e}\| \leq B(\e)$ and $\left\|\frac{\p \sigma_{\delta,\e}}{\p \theta}\right\| \leq \Delta(\e)$ given by \Cref{thm.henry}, we conclude that $\|\sigma_{\delta,\e}\|_{C^1} \to 0$ uniformly on $\delta$ as $\e \to 0^+$.
\end{proof}

\section{Proof of \Cref{thm.main.wnhim}} \label{sec:proofthmA}
The proof of \Cref{thm.main.wnhim} can be summarized in four steps: first, \eqref{eq:thm.main1} is embedded in a two parameter family devised to allow the application of the method of continuation described in \Cref{sec:Continuation.General.Result}; secondly, a change of variables is applied to put this family in the form \eqref{eq:systemtransformedfinal}, which can be studied with the tools of \Cref{sec:existingresultshenry}; next, \Cref{thm.henry} is applied to obtain a family $\mathcal{C}_{\delta,\e}$ of invariant manifolds of \eqref{eq:systemtransformedfinal}; finally, returning to the original coordinates, this family of invariant manifolds becomes a family of tori $\mathcal{T}_{\delta,\e}$ satisfying \ref{hyp.P-thmGC}, so that \Cref{thm:GC} may be applied to obtain normally hyperbolic invariant tori of \eqref{eq:thm.main1}, as wished.

\subsection{The setting}
We investigate system \eqref{eq:thm.main1} assuming that hypotheses \ref{hyp.1-thm1} and \ref{hyp.2-thm1} hold. As mentioned above, before rewriting it in a form that allows the application of \Cref{thm.henry}, we must embed it in the two parameter family introduced in \Cref{sec:Continuation.General.Result}. Hence, we will study
\begin{equation} \label{eq:continuationfamilyproof}
	\dot x = \delta^\ell f(x) + \e^{\ell+1} h(t,x,\e).
\end{equation}

\subsection{Change of coordinates}

We employ an adapted coordinate system defined in a neighborhood $\mathcal{U}_0$ of the invariant torus $\mathcal{T}_0$ to write \eqref{eq:continuationfamilyproof} in a form to which \Cref{thm.henry} can be applied to obtain a family of invariant tori $\mathcal{T}_{\delta,\e}$. We will thus be able to verify hypothesis \ref{hyp.P-thmGC} of \Cref{thm:GC}, from which normal hyperbolicity of $\mathcal{T}_\e : = \mathcal{T}_{\e,\e}$ will be concluded.

The above-mentioned coordinate system comprises two sets of coordinates: the angular variables $\theta =(\theta_1,\ldots,\theta_d) \in \R^d$, essentially describing movement along $\mathcal{T}_0$, and ${\bf h} = ({\bf h}_1,\ldots,{\bf h}_{n-d}) \in \R^{n-d}$, which describe the remaining $(n-d)$ transversal directions. Periodicity will always be implied for the variables $\theta_i$ due to their angular nature. Hence, we will conveniently adopt the notation $\mathbb{T}^d$ as the topological $d$-torus, i.e., the Cartesian product of $d$ copies of $\cc^1$.

To define the coordinates $(\theta, {\bf h})$, we begin by observing that $\mathcal{T}_0$ is of class $C^3$. 
Since it is also a $d$-torus, there is a $C^{3}$ embedding $\overline{A}:\mathbb{T}^d \to \R^n$ whose image is $\mathcal{T}_0$. Hence, we obtain a $C^{3}$ function $A:  \R^d \to \R^n$ that is $1$-periodic in each of its entries whose image is $\mathcal{T}_0$. We can assume that $A$ is injective when restricted to $[0,1)^d$.

Triviality of the normal bundle of $\mathcal{T}_0$, on the other hand, ensures that there is a $C^{3}$ function $Q: \R^d \to \mathcal{L}(\R^{n-d},\R^n)$ such that, for each $\theta \in \R^d$, the operator $Q(\theta)$ is injective and satisfies
\begin{equation*} \label{eq:splittingQ}
	T_{A(\theta)} \mathcal{T}_0 \oplus \text{Range} (Q(\theta)) = \R^n.
\end{equation*}
The columns of $Q(\theta)$ are elements of a global frame on the normal bundle of $\mathcal{T}_0$, which is guaranteed to exist on account of triviality.

The function mapping the new coordinates $(\theta, {\bf h})$ to the old ones via $x = P(\theta,{\bf h})$ is defined by
\begin{equation*}
	P: (\theta, {\bf h}) \mapsto A(\theta) + Q(\theta) \cdot {\bf h}.
\end{equation*}
Evidently, $P$ is not a diffeomorphism, as each $p \in \mathcal{T}_0$ admits infinitely many representations with equivalent angular variables. However, $P$ provides a suitable coordinate system as a covering map from $\R^d \times \mathcal{B}_{n-k}(0,\rho_0)$, with $\rho_0 >0$ sufficiently small, to a neighborhood $\mathcal{U}_0$ of $\mathcal{T}_0$. This follows from the fact that $A(\theta)$ is an embedding and that the columns of $Q(\theta)$ constitute a global frame.

Lifting \eqref{eq:thm.main1} via the map $P$, we obtain
\begin{align}\label{eq:transformedpullback}
	\left[\begin{array}{c}
		\dot \theta \\
		\mathbf{\dot h}
	\end{array}\right] =\delta^\ell \left(DP(\theta,\mathbf{h})\right)^{-1} \left[ f (P(\theta,\mathbf{h})) + \frac{\e^{\ell+1}}{\delta^\ell} h (t,P(\theta,\mathbf{h}),\e)\right].
\end{align}
Considering that $(\theta,{\bf h}) \mapsto f(P(\theta,{\bf h}))$ is of class $C^3$ and $(\theta,{\bf h}) \mapsto (DP(\theta,{\bf h}))^{-1}$ is of class $C^2$, we can expand them near ${\bf h} =0$, obtaining
\begin{equation*} \label{eq:expansionsgellDP}
	\begin{aligned}
		&f (P(\theta,\mathbf{h})) = f (A(\theta)) + Df(A(\theta))\cdot Q(\theta) \cdot \mathbf{h} + R_f(\theta,\mathbf{h}), \\
		&\left(DP(\theta,\mathbf{h})\right)^{-1} = \left(	DP(\theta,0)\right)^{-1} + \frac{\p \left(DP\right)^{-1}}{\p \mathbf{h}}(\theta,0) \cdot \mathbf{h} + R_P(\theta,\mathbf{h}),
	\end{aligned} 
\end{equation*}
where $R_f(\theta,0)$, $R_P(\theta,0)$, $\frac{\p R_f}{\p \mathbf{h}}(\theta,0)$, and $\frac{\p R_P}{\p \mathbf{h}}(\theta,0)$ all vanish for any $\theta \in \R^d$. 

Since $\mathcal{T}_0$ is invariant under the flow of (\ref{eq:thm.main1}), it follows that ${\bf \dot h} =0$ if we take ${\bf h}=0$ in (\ref{eq:transformedpullback}). Hence, there is $w:\R^d \to \R^d$ such that 
\begin{align}\label{eq:defw(theta)}
	\left(DP(\theta,0)\right)^{-1} \cdot f (A(\theta)) =  \left[\begin{array}{c}
		w(\theta) \\
		0
	\end{array}\right].
\end{align}
Moreover, by definition of $P$, if $D$ denotes the Jacobian derivative, we have
\begin{equation*}
	DP(\theta,{\bf h}) = \left[DA(\theta) + \frac{\p}{\p \theta}(Q(\theta) \cdot {\bf h}) \; \Big| \; Q(\theta) \right],
\end{equation*}
where $\frac{\p}{\p \theta}(Q(\theta) \cdot {\bf h})$ is the $n \times d$ matrix function of class $C^2$ given by
\begin{equation*}
	\frac{\p}{\p \theta}(Q(\theta) \cdot {\bf h}) = \left[ \frac{\p Q}{\p \theta_1} (\theta) \cdot {\bf h} \; \; \ldots \; \; \frac{\p Q}{\p \theta_d} (\theta) \cdot {\bf h}\right].
\end{equation*}
Differentiating $DP(\theta,{\bf h})$ with respect to ${\bf h}$ yields
\begin{align*}
	\frac{\p \left(DP\right)}{\p \mathbf{h}}(\theta,0) \cdot {\bf h} = \left[\begin{array}{c|c}
		\frac{\p}{\p \theta}(Q(\theta) \cdot {\bf h})& 0
	\end{array}\right].
\end{align*}
Thus, it follows by differentiating the identity $(DP(\theta,{\bf h})) \cdot (DP(\theta, {\bf h}))^{-1} = \Id$ with respect to ${\bf h}$ at ${\bf h} = 0$ that
\begin{align*}
	\frac{\p \left(DP\right)^{-1}}{\p \mathbf{h}}(\theta,0) \cdot \mathbf{h} = - 	\left(DP(\theta,0)\right)^{-1} \cdot \left[\begin{array}{c|c}
		\frac{\p}{\p \theta}(Q(\theta) \cdot {\bf h}) & 0
	\end{array}\right] 	\cdot \left(DP(\theta,0)\right)^{-1}.
\end{align*}

Expanding the product of the two equations in \cref{eq:expansionsgellDP} and taking \cref{eq:defw(theta)} into account, we obtain $S(\theta)$, $H(\theta)$, $R_\theta(\theta,{\bf h})$, and $R_{\bf h}(\theta,{\bf h})$ such that
\begin{align} \label{eq:transformationR}
	\left(DP(\theta,\mathbf{h})\right)^{-1}\cdot f (P(\theta,\mathbf{h})) = \left[\begin{array}{c}
		w(\theta) \\
		0
	\end{array}\right] +  \left[\begin{array}{c}
		S(\theta) \cdot \mathbf{h} \\
		H(\theta) \cdot \mathbf{h} 
	\end{array}\right]+  \left[\begin{array}{c}
		R_{\theta}(\theta, \mathbf{h}) \\
		R_{{\bf h}}(\theta, \mathbf{h}),
	\end{array}\right]
\end{align}
and $R_\theta(\theta,0)$, $R_{{\bf h}}(\theta,0)$, $\frac{\p R_\theta}{\p \mathbf{h}}(\theta,0)$, and $\frac{\p R_{{\bf h}}}{\p \mathbf{h}}(\theta,0)$ all vanish for any $\theta \in \R^d$.
One can verify that $S(\theta) \cdot \mathbf{h}$ and $H(\theta) \cdot \mathbf{h}$ are respectively given by the the projection onto the first $d$ and last $n-d$ entries of
\begin{equation*}\label{eq:defSH}
	\begin{aligned}
		\left(\big(DP(\theta,0)\big)^{-1}\cdot Df(A(\theta)) \cdot Q(\theta)  -  \big(DP(\theta,0)\big)^{-1} \cdot \big(DQ(\theta) \cdot w(\theta)\big)\right)\cdot \mathbf{ h} .
	\end{aligned}
\end{equation*}

Similarly, we define the functions $Z$ and $W$ as the first $d$ and last $n-d$ entries of the product $\big(DP(\theta,\mathbf{h})\big)^{-1} \cdot h (t,P(\theta,\mathbf{h}),\e)$, i.e., 
\begin{align}\label{eq:defZW}
	\big(DP(\theta,\mathbf{h})\big)^{-1} \cdot h (t,P(\theta,\mathbf{h}),\e) = \left[\begin{array}{c}
		Z(t,\theta,\mathbf{h},\e) \\
		W(t,\theta,\mathbf{h},\e)
	\end{array}\right].
\end{align} 
We also define $\Theta$ and $\zeta$ directly by
\begin{equation*} \label{eq:defThetazeta}
	\begin{aligned}
		&\Theta(t,\theta,\mathbf{h},\delta,\e): =S(\theta) \cdot \mathbf{h}+ R_{\theta}(\theta, \mathbf{h}) + \frac{\e^{\ell+1}}{\delta^\ell} Z\left(\frac{t}{\delta^\ell},\theta,\mathbf{h},\e\right), \\
		&\zeta(t,\theta,\mathbf{h},\delta,\e): = R_{\mathbf{h}}(\theta, \mathbf{h}) + \frac{\e^{\ell+1}}{\delta^\ell} W\left(\frac{t}{\delta^\ell},\theta,\mathbf{h},\e\right),
	\end{aligned}
\end{equation*}
which are easily seen to be $\delta^\ell T$-periodic in $t$. Finally, re-scaling time $t \to \delta^\ell t$ in \cref{eq:transformedpullback}, it becomes
\begin{equation}\label{eq:systemtransformedfinal}
	\begin{aligned}
		& \dot \theta = w(\theta) + \Theta(t,\theta,\mathbf{h},\delta,\e), \\
		& \mathbf{\dot h} = H(\theta) \cdot \mathbf{h} + \zeta(t,\theta,\mathbf{h},\delta,\e),
	\end{aligned}
\end{equation}
which is in the form required for the application of \Cref{thm.henry}. In what follows, we show that it satisfies the hypotheses of this theorem. 

\subsubsection{Properties of the transformed system}

We will verify that \eqref{eq:systemtransformedfinal} satisfies the hypotheses of \Cref{thm.henry}. We start by recalling that the investigation is performed on parameter spaces of the form
\begin{equation*}
	E(\delta_0,\e_0) = \{(\delta,\e) \in (0,\delta_0) \times (0,\e_0):\e \leq \delta \}, \; \text{where} \; \delta_0>0, \; \e_0>0.
\end{equation*}
The first hypothesis, \ref{hyp.thmhenry.thetaperiodic}, follows easily from the fact that $A(\theta)$ and $Q(\theta)$ are $1$-periodic in each entry. For \ref{hyp.thmhenry.C2andbounded}, we will need the following lemma.

\begin{lemma} \label{lemma.regularity.transformedsystem}
	$w$, $R_\theta$, $R_\mathbf{ h}$, $Z$, $W$, and the functions  $(\theta,\mathbf{h}) \mapsto S(\theta) \cdot \mathbf{ h}$ and $(\theta,\mathbf{h}) \mapsto H(\theta) \cdot \mathbf{h}$ are of class $C^{2}$ for $(t,\theta,\mathbf{h},\e) \in \R \times \R^d \times \mathcal{B}_{n-d}(0,\rho_0) \times (-\e_0,\e_0)$. Moreover, $\Theta$ and $\zeta$ are of class $C^{2}$ on the domain $(t,\theta,\mathbf{h},\delta,\e) \in \R \times \R^d \times \mathcal{B}_{n-d}(0,\rho_0) \times E(\e_0,\e_0)$.
\end{lemma}
\begin{proof}
	$P$ is of class $C^{3}$ because $A$ and $Q$ are of class $C^3$. Hence, $w$ is of class $C^{2}$ by \cref{eq:defw(theta)}. Similarly, $(\theta,\mathbf{h}) \mapsto S(\theta) \cdot \mathbf{ h}$ and $(\theta,\mathbf{h}) \mapsto H(\theta) \cdot \mathbf{h}$ are of class $C^{2}$ by \cref{eq:defSH}. It then follows from \cref{eq:transformationR} that $R_\theta$ and $R_\mathbf{h}$ are also of class $C^{2}$. Furthermore, \cref{eq:defZW} ensures that $Z$ and $W$ are of class $C^{2}$ as well. Finally, since $(\delta,\e) \in E(\e_0,\e_0)$ excludes the possibility of $\delta=0$, \cref{eq:defThetazeta} guarantees that $\Theta$ and $\zeta$ are of class $C^{2}$ in that domain.
\end{proof}

The next lemma takes care of verifying \ref{hyp.thmhenry.uniformbound} and \ref{hyp.thmhenry.perturbationsize}.

\begin{lemma}
	Let $\rho_1 \in (0, \rho_0)$ and $\e_1 \in (0,\e_0)$. There is $M(\rho_1,\e_1)>0$ such that every function appearing in \cref{eq:systemtransformedfinal} is bounded by $M(\rho_1,\e_1)$ for $(\tau,\theta,\mathbf{h},\delta,\e) \in \R \times \R^d \times \overline{\mathcal{B}}_{n-d}(0,\rho_1) \times E(\e_0,\e_1)$. Moreover, there is $M_0(\e_1)>0$ such that $\Theta$ and $\zeta$ satisfy the inequalities in \ref{hyp.thmhenry.perturbationsize} with $q(\e) = \e M_0(\e_1)$.
\end{lemma}
\begin{proof}
	$Z$ and $W$ are $T$-periodic in their first entries, and periodic in each component of $\theta$. The same holds for the partial derivatives up to the second order of $Z$ and $W$ with respect to $\theta$ and ${\bf h}$. Thus, since $Z$ and $W$ are continuous, restricting $\e$ to the compact $[-\e_1,\e_1]$, it follows that there is $M_0(\e_1)> 0$ such that
	\begin{align*}
		\sum_{0\leq i_1+i_2 \leq 2} \left\|\frac{\p^{i_1+i_2}Z}{\p \theta^{i_1} \p {\bf h}^{i_2}} (t,\theta,0,\e)\right\|+\left\|\frac{\p^{i_1+i_2}W}{\p \theta^{i_1} \p {\bf h}^{i_2}} (t,\theta,0,\e)\right\| \leq M_0(\e_1),
	\end{align*}
	for $(t,\theta,\e) \in \R \times \R^d \times [-\e_1,\e_1]$. In particular, considering the definitions of $\Theta$ and $\zeta$ and the fact that $\e \leq \delta$ in $E(\e_0,\e_1)$, it follows that, for any chosen $\e_* \in (0,\e_1]$,
		\begin{align*}
				\|\Theta(t,\theta,0,\delta,\e)\|+\|\zeta(t,\theta,0,\delta,\e)\| \leq \e_* M_0(\e_1),
			\end{align*}
		for $(t,\theta,\delta,\e) \in \R \times \R^d \times E(\e_0,\e_*)$. An analogous argument establishes the same bound for the derivatives of $\Theta$ and $\zeta$ with respect to $\theta$, and of $\zeta$ with respect to ${\bf h}$, as stated in \ref{hyp.thmhenry.perturbationsize}. 
	
	Due to periodicity and compactness, a similar argument to the one above ensures the existence of $M(\rho_1,\e_1)>0$ such that
	\begin{align*}
		\sum_{0\leq i_1+i_2 \leq 2} \left\|\frac{\p^{i_1+i_2}Z}{\p \theta^{i_1} \p {\bf h}^{i_2}} (t,\theta,{\bf h},\e)\right\|+\left\|\frac{\p^{i_1+i_2}W}{\p \theta^{i_1} \p {\bf h}^{i_2}} (t,\theta,{\bf h},\e)\right\| \leq M(\rho_1,\e_1),
	\end{align*}
	for $(t,\theta,\mathbf{h},\e) \in \R \times \R^d \times \overline{\mathcal{B}}_{n-d}(0,\rho_1) \times [-\e_1,\e_1]$. In very much the same fashion, observe that $w$, $S$, $R_\theta$, $H$ and $R_\mathbf{h}$ are all continuous on $(\theta,\mathbf{h}) \in \R^d \times \overline{\mathcal{B}}_{n-d}(0,\rho_1)$ and periodic in each entry for $\theta$. Thus, they must all be uniformly bounded on this domain, and it is easy to see that this implies that $\Theta$ and $\zeta$ must be uniformly bounded on $(\tau,\theta,\mathbf{h},\delta,\e) \in \R \times \R^d \times \overline{\mathcal{B}}_{n-d}(0,\rho_1) \times E(\e_0,\e_1)$. By potentially redefining $M(\rho_1,\e_1)$, we can assume it is such a uniform bound.
\end{proof}

Reasoning as we did above but with $h \equiv 0$, one can easily verify that, in the adapted coordinate system we built, the equation $\dot x =f(x)$ becomes
\begin{equation}\label{eq:pullbackunperturbed}
	\begin{aligned}
		&\dot \theta = w(\theta) + S(\theta)\cdot \mathbf{h} + R_\theta(\theta,\mathbf{h}), \\
		&\mathbf{\dot h} = H(\theta)\cdot  \mathbf{h}	+ R_\mathbf{h}(\theta,\mathbf{h}).
	\end{aligned}
\end{equation}
Henceforth, we denote by $\phi_t(x)$ the flow of the original system $\dot x = f(x)$ and by $\psi_t(\theta,{\bf h})$ the flow of \eqref{eq:pullbackunperturbed}, in adapted coordinates. By nature of the coordinates $(\theta,\mathbf{h})$, $\mathcal{T}_0$ corresponds to the set $\mathcal{C}_0=\{(\theta,0): \theta \in \R^k\}$. Thus, the transformed equations restricted to this invariant manifold are $\dot \theta =w(\theta)$, $\mathbf{\dot h} =0$. 

As in \Cref{sec:existingresultshenry}, let $Y(t,t_0,\theta_0)$ be the solution of $\dot \theta = w(\theta)$ for which $\theta(t_0) = \theta_0$. The first variational equation of \cref{eq:pullbackunperturbed} along its solution $(Y(t,t_0,\theta_0),0,0) \in \mathcal{C}_0$ is
\begin{equation} \label{eq:firstvariationalunperturbed}
	\begin{aligned}
		&\dot u = Dw(Y(t,t_0,\theta_0))\cdot u +S(Y(t,t_0,\theta_0)) \cdot {\mathbf v} , \\
		&\dot {\mathbf v} = H(Y(t,t_0,\theta_0))\cdot {\mathbf v}, \\
	\end{aligned} 
\end{equation}
with $u$ providing the tangential direction and ${\mathbf v}$ the attracting normal direction. From now on, we follow the notation established in \ref{hyp.thmhenry.normalbehavior}, denoting the principal fundamental matrix solution of $\dot {\mathbf v} = H(Y(t,t_0,\theta_0))\cdot {\mathbf v}$ by $\Phi(t,s,t_0,\theta_0)$. 

We will show that absolute normal hyperbolicity of $\mathcal{T}_0$ implies \ref{hyp.thmhenry.tangentbehavior} and \ref{hyp.thmhenry.normalbehavior}. Recall that, by definition, the fact that $\mathcal{T}_0$ is absolutely $r_H$-normally hyperbolic is equivalent (see \Cref{sec:NH}) to there being $K_A,\, K_B>1$ and $\rho_M,\, \rho_N >0$ such that $\rho_N > r_H \rho_M$ and, for all $p \in \mathcal{T}_0$, the following hold
\begin{equation} \label{eq:ANH-conditions}
	\| A_t(p) \| \leq K_A e^{\,\rho_M |t|}, \; \forall t \in \R \quad \text{and} \quad \|B_t(p)\| \leq K_B e^{-\rho_N t }, \; \forall t>0.
\end{equation} 

Since the restricted flow $\phi_{t-t_0}|_{\mathcal{T}_0}(x_0)$ corresponds, in adapted coordinates, to $$ \psi_{t-t_0}(\theta_0,0) = (Y(t,t_0,\theta_0),0) \in \R^d \times \R^{n-d},$$
where $x_0 = P(\theta_0,0)$ and $t,\, t_0 \in \R$ can be chosen arbitrarily, the first condition in \eqref{eq:ANH-conditions} can be rewritten as 
\begin{equation} \label{eq:pYptheta0bound}
	\sup_{\theta_0 \in \R^d} \left\| \frac{\p Y}{\p \theta_0}(t,t_0, \theta_0) \right\| \leq K'_A e^{\rho_M|t-t_0|},
\end{equation}
for any $t,t_0 \in \R$, where $K'_A$ is a modified constant to take into account metric distortions introduced by the coordinate lift. In order to prove that \ref{hyp.thmhenry.tangentbehavior} holds, fix $t, \, t_0 \in \R$ and let $\theta_1, \, \theta_2 \in \R^d$ be given. Define $G_{t,t_0}:[0,1] \to \R^d$ by
\begin{equation*}
	G_{t,t_0}(s) = Y(t,t_0,\theta_2 + s(\theta_1-\theta_2)).
\end{equation*}
Then, it follows from \eqref{eq:pYptheta0bound} that
\begin{equation*}
	\|G_{t,t_0}'(s)\| = \left\|\frac{\p Y}{\p \theta_0} (t,t_0,\theta_2+s(\theta_1-\theta_2)) \cdot (\theta_1-\theta_2)\right\| \leq K'_A e^{\rho_M|t-t_0|} \|\theta_2-\theta_1\|.
\end{equation*}
Therefore, considering that $Y(t,t_0,\theta_1) = G_{t,t_0}(1)$ and $Y(t,t_0,\theta_2) = G_{t,t_0}(0)$, the Fundamental Theorem of Calculus ensures that
\begin{equation*}
	\|Y(t,t_0,\theta_1) - Y(t,t_0,\theta_2)\| \leq \int_0^1 \|G'(s)\| ds \leq K'_A e^{\rho_M|t-t_0|} \|\theta_2-\theta_1\|,
\end{equation*}
proving \ref{hyp.thmhenry.tangentbehavior}.

To prove \ref{hyp.thmhenry.normalbehavior}, we first observe that, by properties of the flow,
\begin{equation} \label{eq:Yflowproperty}
	(Y(t,t_0,\theta_0),0) = \psi_{t-t_0}(\theta_0,0) = \psi_{t-s} \circ \psi_{s-t_0}(\theta_0,0) = (Y(t,s,Y(s,t_0,\theta_0)),0).
\end{equation}
Now, if we define $(u(t,s,t_0,\theta_0),{\bf v}(t,s,t_0,\theta_0)) := D\psi_{t-s}(Y(s,t_0,\theta_0),0)$, it follows by changing order of derivatives that $(u,{\bf v})$ satisfies the differential system
\begin{equation*} \label{eq:firstvariationalunperturbedcomposition}
	\begin{aligned}
		&\dot u = Dw(Y(t,s,Y(s,t_0,\theta_0)))\cdot u +S(Y(t,s,Y(s,t_0,\theta_0))) \cdot {\mathbf v} , \\
		&\dot {\mathbf v} = H(Y(t,s,Y(s,t_0,\theta_0)))\cdot {\mathbf v}, \\
	\end{aligned} 
\end{equation*}
which, on account of \eqref{eq:Yflowproperty}, is equivalent to \eqref{eq:firstvariationalunperturbed}. Moreover, if $\Pi_{(n-d)}:\R^n \to \R^{(n-d)}$ and $\iota_{(n-d)}: \R^{(n-d)} \to \R^n$ denote, respectively, the projection and corresponding inclusion onto the last $n-d$ coordinates, it is clear that 
\begin{equation*}
	\Pi_{(n-d)} \cdot D\psi_{s-s} (Y(s,t_0,\theta_0),0) \cdot \iota_{(n-d)} = \Id \in \R^{(n-d) \times (n-d)},
\end{equation*}
so that both $\Pi_{(n-d)} \cdot D\psi_{t-s} (Y(s,t_0,\theta_0),0) \cdot \iota_{(n-d)}$ and $\Phi(t,s,t_0,\theta_0)$ are solutions of the same initial value problem. Hence, 
\begin{equation*}
	\Pi_{(n-d)} \cdot D\psi_{t-s} (Y(s,t_0,\theta_0),0) \cdot \iota_{(n-d)} = \Phi(t,s,t_0,\theta_0).
\end{equation*}

As remarked in \Cref{sec:NH}, normal hyperbolicity is independent of the choice of metric. Hence, we can choose a metric over $\mathcal{T}_0$ that considers the fibers generated by the global frame $Q(\theta)$ of the normal bundle as everywhere orthogonal to the tangent space of $\mathcal{T}_0$. In that case, the projection $\Pi$ onto the normal bundle $N$ of $\mathcal{T}_0$ corresponds to $\Pi_{(n-d)}$ in the adapted coordinates, so that $\Pi_{(n-d)} \cdot D\psi_{t-s} (Y(s,t_0,\theta_0),0) \cdot \iota_{(n-d)}$ itself corresponds to 
\begin{equation*}
	\Pi \cdot D\phi_{t-s} \big(\phi_{s-t_0}\big(P(\theta_0,0)\big)\big)\big|_{N}.
\end{equation*}
Moreover, the second condition in \eqref{eq:ANH-conditions} guarantees that, for $t \geq s$,
\begin{equation*}
	\sup_{p \in \mathcal{T}_0} \left\| \Pi \cdot D\phi_{t-s} (p)\big|_{N}\right\| \leq K_B e^{-\rho_N (t-s)}.
\end{equation*}
Thus, by modifying the constant to take into account the coordinate lift, it follows that
\begin{equation*}
	\|\Phi(t,s,t_0,\theta_0)\| \leq \sup_{\theta_0 \in \R^d} \left\|\Pi_{(n-d)} \cdot D\psi_{s-s} (Y(s,t_0,\theta_0),0) \cdot \iota_{(n-d)}\right\| \leq K'_B e^{-\rho_N(t-s)},
\end{equation*}
for $t \geq s$, proving \ref{hyp.thmhenry.normalbehavior}.

\subsection{Invariant manifolds for the adapted system}
We have proved that \eqref{eq:systemtransformedfinal} satisfies hypotheses \ref{hyp.thmhenry.thetaperiodic}---\ref{hyp.thmhenry.perturbationsize} with $\alpha = \rho_N$ and $\beta = \rho_M$. By \ref{hyp.2-thm1}, the torus $\mathcal{T}_0$ is absolutely normally hyperbolic, so that $\beta = \rho_M < \rho_N = \alpha$, and we can apply \Cref{thm.henry} with an appropriately chosen $\mu>0$ to find $\e_1 \in (0,\e_0)$, as well as $\mathcal{C}_{\delta,\e}$ and $\sigma_{\delta,\e}$ as described in the theorem.

Thus, if $(\delta,\e) \in E(\e_0,\e_1)$, the manifold $\mathcal{C}_{\delta,\e} = \{(\theta,\sigma_{\delta,\e}(t,\theta)) : t \in \R, \, \theta \in \R^d\}$ is invariant under \eqref{eq:systemtransformedfinal}. Equivalently, the manifold 
$$\mathcal{C}'_{\delta,\e} = \{(s,\theta,\sigma_{\delta,\e}(s,\theta)) : s \in \R, \, \theta \in \R^d\} \subset \R^{1+n}$$
is invariant under the flow of the autonomous differential system
\begin{equation*}
	\begin{aligned}
	& \dot s = 1, \\
	& \dot \theta = w(\theta) + \Theta(s,\theta,\mathbf{h},\delta,\e), \\
	& \mathbf{\dot h} = H(\theta) \cdot \mathbf{h} + \zeta(s,\theta,\mathbf{h},\delta,\e).
\end{aligned}
\end{equation*}
\Cref{thm.henry} guarantees that $\sigma_{\delta,\e}$ is $1$-periodic in each entry of $\theta$ and $\delta^\ell T$-periodic in its first entry. Moreover, \Cref{lemma.sigma.tdiff} ensures that $\sigma_{\delta,\e}$ is of class $C^1$ and that $\|\sigma_{\delta,\e}\|_{C^1} \to 0$ uniformly on $\delta$ as $\e \to 0^+$.

\subsection{Returning to the original coordinates and finishing the proof}

The final step of the proof is returning $(\theta,{\bf h})$ to the original coordinates $x$, after which $\mathcal{C}'_{\delta,\e}$ becomes the $C^1$ manifold
$$
\mathcal{M}_{\delta,\e} = \big\{\big(s \mod \delta^\ell T,\, A(\theta) + Q(\theta)\cdot \sigma_{\delta,\e}(s, \theta)\big): s \in \R, \, \theta \in \R^d \big\},
$$
invariant under the flow of 
\begin{equation*} \label{eq:thm.proofmain.finalde}
	\begin{aligned}
		&\dot s = 1, \quad
		&\dot x = f(x) + \frac{\e^{\ell+1}}{\delta^\ell} h\left(\frac{s}{\delta^\ell},x,\e\right).
	\end{aligned}
\end{equation*}
This return is well-behaved with respect to the covering map employed in the coordinate change on account of the established periodicity of $\sigma_{\delta,\e}$ with respect to $\theta$.

The first two hypotheses of \Cref{thm:GC} are certainly valid for \eqref{eq:thm.main1} on account of \ref{hyp.1-thm1} and \ref{hyp.2-thm1}. Furthermore, since $\sigma_{\delta,\e}$ is $C^1$ and $\|\sigma_{\delta,\e}\|_{C^1} \to 0$ uniformly on $\delta$ as $\e \to 0^+$, it follows easily that hypothesis \ref{hyp.P-thmGC} of the same theorem also holds for the family $\mathcal{M}_{\delta,\e}$ constructed above. Hence, we are justified in applying \Cref{thm:GC}, potentially having to redefine $\e_1$, to obtain, for each $\e_1 \in (0,\e_0)$ an attracting, absolutely $r_H$-normally hyperbolic, compact invariant manifold $\mathcal{N}_\e =: \mathcal{T}_\e^{d+1}$ for
\begin{equation*}
	\dot \tau =1, \quad \dot x =\e^\ell f(x) +\e^{\ell+1} h(\tau,x,\e),
\end{equation*}
defined on the quotient space $\R / ( \mathbb{Z} T) \times \R^n$.

It remains to prove that $\mathcal{N}_\e$ is a $(d+1)$-torus satisfying the statements given in \Cref{thm.main.wnhim}. This follows from the way $\mathcal{N}_\e$ is constructed in the first paragraph of the proof of \Cref{thm:GC}. In fact, it is stated therein that, by re-scaling time $t\to \delta^\ell t$ and changing variables via $s = \delta^\ell \tau $, one obtains the family $\mathcal{N}_{\delta,\e}$ from $\mathcal{M}_{\delta,\e}$. In the case we are studying, we can use the explicit expression of $\mathcal{M}_{\delta,\e}$ to obtain
$$\mathcal{N}_{\delta,\e} = \big\{\big(\tau \mod T,\, A(\theta) + Q(\theta)\cdot \sigma_{\delta,\e}(\delta^\ell \tau , \theta)\big): \tau \in \R, \, \theta \in \R^d \big\} \subset \cc^1 \times \R^{n},$$
where $\cc^1$ denotes the quotient space $\R / (T\Z)$. Periodicity of $\sigma_{\delta,\e}$ ensures that each $\mathcal{N}_{\delta,\e}$ is a $(d+1)$-torus, given in the $(\tau,x) \in \cc^1 \times \R^n$ coordinate system by the relation $x = \sigma_{\delta,\e} (\delta^\ell \tau, \theta)$.

Also following the proof of \Cref{thm:GC}, $\mathcal{N}_\e$ is defined to be $\mathcal{N}_{\e,\e}$, which must be of class $C^{r_H}$ because it is a $C^1$ invariant manifold which is $r_H$-normally hyperbolic (see, for instance, \cite[Theorem 4.1]{hirschpughshub}). This proves statement \ref{item.thm.main.wnhim.smooth} of \Cref{thm.main.wnhim}. 
 
Statements \ref{item.thm.main.wnhim.F} and \ref{item.thm.main.wnhim.periodic} follow straightforwardly from defining $\mathcal{F}_\e(\tau,\theta) : = A(\theta) +Q(\theta) \cdot \sigma_{\e,\e} (\e^\ell \tau, \theta) $, which is easily verified to be $T$-periodic in $\tau$ and $1$-periodic in each entry of $\theta$. Moreover, since $\|\sigma_{\delta,\e}\|_{C^1} \to 0$ uniformly on $\delta$ as $\e \to 0^+$, it follows that $\mathcal{F}_\e$ as just defined must converge in the $C^1$ norm to $(\tau,\theta) \mapsto A(\theta)$, proving \ref{item.thm.main.wnhim.Fconvergence}. Finally, statement \ref{item.thm.main.wnhim.Tconvergence} follows directly from \ref{item.thm.main.wnhim.Fconvergence}.

To conclude the proof of \Cref{thm.main.wnhim}, it remains only to argue that $\mathcal{T}_\e^{d+1}$ must have trivial normal bundle. This follows directly from \ref{item.thm.main.wnhim.F} and \ref{item.thm.main.wnhim.Fconvergence}. In fact, those properties guarantee that $\mathcal{T}_\e^{d+1}$ can be seen as the graph of a smooth section of the normal bundle of $\cc^1 \times \mathcal{T}_0^d$, which is trivial on account of being the pullback of the trivial normal bundle of $\mathcal{T}_0^d$ via the projection onto the second factor.

\section{Proof of Theorems \ref{thm.main.avg} and \ref{thm:app}} \label{sec:proofsBC}

\subsection{Proof of \Cref{thm.main.avg}} \label{sec:proofThB}
As noted in \Cref{sec:IntroAveraging}, there is a $T$-periodic change of variables \eqref{eq:AvgCofV} taking \eqref{system.standardform.intro} into \eqref{system.standardform.intro.avg}, which is in the form to which \Cref{thm.main.wnhim} can be applied. Also, it is easy to see that \ref{hyp.1-thm1} holds on account of periodicity in $t$ of the functions in \eqref{system.standardform.intro}. Finally, the assumption that the guiding system $\dot z = g_\ell(z)$ has an attracting $r_H$-normally hyperbolic invariant $d$-torus $\mathcal{T}_0^d$ corresponds directly to \ref{hyp.2-thm1}.

Therefore, an application of \Cref{thm.main.wnhim} guarantees the existence of a $(d+1)$-torus in the extended phase space of \eqref{system.standardform.intro.avg}, which becomes the $(d+1)$-torus $\mathcal{T}_\e^{d+1}$ of \eqref{system.standardform.intro.ext} after reverting the change of variables given by \eqref{eq:AvgCofV}. The fact that this $(d+1)$-torus is of class $C^{r_H}$ follows directly from its being $r_H$-normally hyperbolic \cite[Theorem 4.1]{hirschpughshub}, and its convergence from property \ref{item.thm.main.wnhim.Tconvergence} of \Cref{thm.main.wnhim}.

\subsection{Proof of \Cref{thm:app}}\label{sec:proofThC}

In cylindrical coordinates $(x,y,Z)=(r\cos\theta,r\sin\theta,Z)$, system \eqref{eq:3dpde} takes the form
\begin{equation*}\label{eq:3dpde2}
	\begin{cases}
		\dot \T=1+\e\,\sin(\T)\cos(\T)P(r^2,Z),\\
		\dot r=\e\, r\sin^2(\T) P(r^2,Z),\\
		\dot Z=2\,\e\, r^2\sin^2(\T) Q(r^2,Z).
	\end{cases}
\end{equation*}
Let $D=(0,\sqrt{b})\times U$. There exists $\varepsilon_0>0$ such that
$\dot \theta>0$ for all $(\theta,r,Z)\in[0,2\pi]\times\overline D$ and
$\varepsilon\in[0,\varepsilon_0]$.
Hence, $\theta$ can be used as a new time variable, yielding
\begin{equation*}\label{eq:napde}
	\begin{cases}
		\dfrac{d r}{d\T}=\dfrac{ \e\,r\sin^2(\T) P(r^2, Z)}{1+\e\,\sin(\T)\cos(\T)P(r^2, Z)}=\e\, r \sin^2(\T)P(r^2, Z)+\e^2R_1(\T,r, Z,\e),\vspace{0.1cm}\\
		\dfrac{ d Z}{d \T}=\dfrac{2\,\e\, r^2\sin^2(\T) Q(r^2,Z)}{1+\e\,\sin(\T)\cos(\T)P(r^2,Z)}=\e\,2r^2 \sin^2(\T) Q(r^2,Z)+\e^2R_2(\T,r,Z,\e).
	\end{cases}
\end{equation*}
This system is in the standard form to be studied by means of the method of averaging described in \Cref{sec:IntroAveraging}. To this end, set
\[
\begin{aligned}
	&T=2\pi,\quad t=\T,\quad \x=(r, Z),\\ &F_1(\T,\x)=\Big( r \sin^2(\T)P(r^2, Z),2r^2 \sin^2(\T) Q(r^2, Z)\Big),\text{ and}\\
	&R(\T,\x,\e)=\Big(R_1(\T,r, Z,\e),R_2(\T,r, Z,\e)\Big).
\end{aligned}
\] 

The corresponding guiding system corresponds to $\ell = 1$ and is given by
\begin{equation}\label{eq:gs}
	(r', Z')=\dfrac{1}{2\pi}\int_0^{2\pi} F_1(\T,\x)\,d\T=\left(\dfrac{1}{2}r P(r^2, Z),r^2Q(r^2, Z)\right),\quad (r, Z)\in D.
\end{equation}
Applying the change of variables $(\rho, Z)=(r^2, Z)$,
which is a diffeomorphism on $D$, \eqref{eq:gs} becomes
\begin{equation}\label{eq:gs2}
	(\rho', Z')=\left(\rho P(r, Z),\rho Q(\rho, Z)\right), \quad (\rho, Z)\in K=(0,b)\times(\alpha,\beta).
\end{equation}
Finally, by re-scaling time to divide the right-hand side
by $\rho$, we conclude that \eqref{eq:gs2},
hence also the guiding system \eqref{eq:gs}, admits an attracting $r_H$-normally
hyperbolic invariant $d$-dimensional torus of class $C^3$ and trivial normal bundle.

The proof is completed by applying \Cref{thm.main.avg}
and reverting the successive changes of variables.

\section{Normally hyperbolic invariant $d$-tori in polynomial systems} \label{sec:ToriPoly}

In 1900, David Hilbert proposed a list of 23 open problems that would shape the development of mathematics in the 20th century \cite{hilbert1900mathematische}. Although some of those problems have since been solved, some still stand, one of which is the 16th problem in the list. Part of this problem is concerned with the number of limit cycles for polynomial vector fields in the plane. We describe it briefly in this section.

Let $\pi(P,Q)$ denote the number of limit cycles of the polynomial differential system $$\dot x=P(x,y), \quad \dot y=Q(x,y)$$ in the plane. The second part of Hilbert's 16th Problem investigates whether a uniform upper bound $H(m)$ for the amount of limit cycles of planar polynomial vector fields exists as a function of their degree $m$. Formally, the question can be expressed as whether the $m$-th Hilbert number
\[
H(m):=\sup\{\pi(P,Q):\deg(P),\, \deg(Q)\leq m\}
\]
is finite. If so, its value is also to be ascertained.

Presently, only $H(1)$ is known, since linear vector fields in the plane do not admit limit cycles. For $m\geq2$, even finiteness of $H(m)$ is still uncertain. Nonetheless, attempts to solve this problem have unquestionably prompted important progress in the qualitative theory of vector fields (see, e.g., \cite{Ilyashenko02}). 

Most of the results obtained concerning $H(m)$ consist in providing asymptotic growth estimates and lower bounds for this number. In \cite{CL95}, it was proved that $H(m)$ grows as fast as $m^2\log m$, which was improved to $(m+2)^2\log (m+2)$ in \cite{HL12}. For small $m$, the best currently known lower bounds are provided in \cite{PT19}. 

As an extension of Hilbert's 16th problem, we submit the question of whether the quantity $N_r^d(m)$ defined in \Cref{sec:IntroNHIT} is finite. This problem was already formally proposed (for $r=1$) by the present authors in \cite[Section 4.2]{NP25}, but here we provide new results concerning the general case of $d$-tori, as well as an important clarification on why we should restrict our attention to $d$-tori in $(d+1)$-dimensional ambient space. We first address the latter issue in \Cref{sec:infinitetori} and then return to the results concerning $N_r^d$ in \Cref{sec:dTori}.

\subsection{Infinitely many high-codimension tori} \label{sec:infinitetori}

For polynomial differential systems of dimension $n\geq 3$, one should not expect a finite upper bound on the number of limit cycles. In fact, infinitely many attracting hyperbolic limit cycles may coexist in such systems. For instance, \cite[Section 4]{champneys1996non} constructs a three-dimensional polynomial vector field $X$ of degree $7$ possessing infinitely many attracting hyperbolic limit cycles. These limit cycles correspond to attracting normally hyperbolic $1$-tori with trivial normal bundle (see \Cref{rm:curves}). Applying \Cref{thm:app} to $X$ then produces arbitrarily many coexisting normally hyperbolic invariant $2$-tori in four-dimensional polynomial systems. Iterating this construction further via the same theorem proves that no uniform upper bound exists for the number of normally hyperbolic invariant $(n-2)$-tori in $n$-dimensional polynomial differential systems.

For tori of codimension $k>2$, i.e., $d$-tori in $\R^n$ with $k=n-d>2$, the polynomial vector field $Y_{n}(\mathbf{x},\mathbf{y}) := (X(\mathbf{x}),-\mathbf{y})$, $\mathbf{y}\in\R^{{n}-3}$, also has degree $7$ and possesses infinitely many attracting hyperbolic limit cycles. Hence, a similar application of \Cref{thm:app} to $Y_n$ establishes the following.

\begin{corollary}\label{cor:unbound}
		Let $n$ and $k$ be integers satisfying $n>k\geq 2$. Then, for each $L>0$, there exists an $n$-dimensional polynomial vector field of degree $9\cdot2^{n-k-1}-2$ for which the number of attracting normally hyperbolic codimension-$k$ invariant tori is greater than $L$. 
	\end{corollary}
	\begin{proof}
		Consider the $(k+1)$-dimensional polynomial vector field $Y_{k+1}$, which has degree $7$ and admits infinitely many attracting normally hyperbolic invariant $1$-tori. Selecting $L$ of these tori, we apply \Cref{thm:app} iteratively $n-k-1$ times, choosing $\e>0$ sufficiently small so that each of the $L$ tori survive in every application of the theorem. This yields a $n$-dimensional polynomial vector field $\mathcal{P}_n$ possessing $L$ attracting normally hyperbolic invariant $(n-k)$-tori. Moreover, it follows by monitoring the increase in degree introduced by the right-hand side of \eqref{eq:3dpde} that $\mathcal{P}_n$ has degree $9\cdot 2^{n-k-1}-2$.
	\end{proof}

	\Cref{cor:unbound} establishes that, for $n>k\geq 2$, the maximal number of normally hyperbolic codimension-$k$ invariant tori in $n$-dimensional polynomial vector fields of degree $m$ is unbounded for each $m\geq 9\cdot 2^{n-k-1}-2$.  Consequently, the problem formulated in \Cref{sec:IntroNHIT} is the only general toroidal extension of Hilbert’s $16$th problem that may admit a nontrivial answer. In other words, it is only meaningful to address the problem of finiteness of the maximal number of normally hyperbolic codimension-$k$ invariant tori for polynomial differential systems as a function of their degree when $k=1$. However, whether this maximal number is already unbounded for degrees below $9\cdot 2^{n-k-1}-2$ remains an open problem.

\subsection{Codimension-$1$ normally hyperbolic invariant tori} \label{sec:dTori}

Let $\bP: \R^d \to \R^d$ be a polynomial vector field such that the differential system $\dot \bx = \bP(\bx)$ admits $L$ absolutely $r$-normally hyperbolic invariant $(d-1)$-tori in its phase space, with $r\geq3$. As $r$-normally hyperbolic tori of codimension $1$, they must each have trivial normal bundle (see \Cref{rm:codimension1}), be of class $C^r$ \cite[Theorem 4.1]{hirschpughshub}, and be strictly attracting in either forward or backward time.

Notice that \Cref{thm.main.wnhim}---and hence its corollaries, Theorems \ref{thm.main.avg} and \ref{thm:app}---remains valid if the assumption that $\mathcal{T}_0^d$ is attracting is replaced by the assumption that it is attracting in backward time, i.e., under time reversal $t \to -t$. Thus, each of those $(d-1)$-tori is subject to the application of \Cref{thm:app} and, by keeping track of the increase in degree between systems \eqref{eq:2dpde} and \eqref{eq:3dpde}, one obtains the following corollary.

\begin{corollary}\label{cor:lowerbound}
	Let $d\geq 2$ and $r\geq3$ be integers. Then, the following relationship holds for each natural number $m\geq 2$:
	\begin{equation*}\label{eq:thm}
		N^d_r(m)\geq N^{d-1}_r\left(\left\lfloor\dfrac{m}{2}\right\rfloor-1\right).
	\end{equation*}
	In particular, from \eqref{eq:ineqN1H} and the fact that $H(2)\geq 4$, we conclude that
	\[
	N^d_r( 2^{d+1}-2)\geq 4.
	\]
\end{corollary}

The asymptotic growth of $N_r^{d}(m)$ was analyzed in \cite[Section~4.2]{NP25}, albeit only for $N^d_1(m)$, which was represented by $N_h^d(m)$ in that reference. Essentially the same argument for general $r \in \N$ allows us to draw the following conclusion: $N_r^{d}(m)$ will grow at least as fast as $m^{d+1}$, provided that there exists a $(d+1)$-dimensional polynomial vector field of some degree $m_0>0$ possessing any number $\tau_0>0$ of absolutely $r$-normally hyperbolic invariant tori. More specifically,
\[
\liminf_{m\to\infty}\frac{N_r^{d}(m)}{m^{d+1}}
\geq
\frac{\tau_0}{\bigl(d+m_0\bigr)^{d+1}}.
\]
Taking Corollary~\ref{cor:lowerbound} and \eqref{eq:relationNrNr'} into account, we can now guarantee this asymptotic growth under no further assumptions.

\begin{corollary}
	For any $r \geq1$ and $d\geq2$, the quantity $N_r^{d}(m)$ grows at least as fast as $m^{d+1}$. More precisely,
	\[
	\liminf_{m\to\infty}\frac{N_r^{d}(m)}{m^{d+1}}
	\geq
	\frac{4}{\bigl(2^{d+1}+d-2\bigr)^{d+1}}.
	\]
\end{corollary}

\section*{Funding Information}

DDN was partially supported by the São Paulo Research Foundation (FAPESP), grants 2024/15612-6 and 2026/03312-3; by the Conselho Nacional de Desenvolvimento Científico e Tecnológico (CNPq), grant 301878/2025-0; and by the Coordenação de Aperfeiçoamento de Pessoal de Nível Superior - Brasil (CAPES), through the MATH-AmSud program, grant 88881.179491/2025-01. PCCRP was supported by the S\~{a}o Paulo Research Foundation (FAPESP) grant 2026/04931-9.

\bibliographystyle{abbrv}
\bibliography{references}

\end{document}